\documentclass[11pt,reqno]{article}
\usepackage{amsmath}
\usepackage{amsmath}
\usepackage{amsthm}
\usepackage{amssymb}
\usepackage{epsfig}

\usepackage{calc}
\usepackage{xspace}
\usepackage{comment}
\usepackage{paralist}

\usepackage{url}
\usepackage{setspace}
\usepackage{rotating}
\usepackage{amsmath,amssymb}
\usepackage{subfigure}
\usepackage{array,arydshln}	
\usepackage{multirow}
\usepackage{booktabs}
\usepackage{color,colortbl}
\usepackage{authblk}
\usepackage{hyperref}
\usepackage{enumitem}
\usepackage{xspace}
\usepackage{comment}
\usepackage{tabularx}	
\usepackage[table]{xcolor}
\usepackage{adjustbox}
\usepackage{float}

\DeclareMathOperator*{\argmin}{argmin}
\DeclareMathOperator*{\argmax}{argmax}

\newtheorem{lemma}{Lemma}
\newtheorem{theorem}{Theorem}
\newtheorem{Prop}{Proposition}

\newtheorem{@remark}{\bf Remark}
\newenvironment{remark}{\begin{@remark}\rm}{\end{@remark}}

\newcommand {\A} {\alpha}
\newcommand {\B} {\beta}
\newcommand{\W}{\omega}
\newcommand {\E} {E}
\newcommand {\T} {\theta}
\newcommand{\HT} {\hat{\theta}_{\A,n}}
\newcommand {\ep} {\epsilon}
\newcommand {\pa} {\partial}
\newcommand {\paa}{\partial^2_{\T\T'}}

\newcommand {\ud}{\mathrm{d}}

\title{ Test for parameter change in the presence of outliers:\\ the density power divergence based approach}
\author[1]{Junmo Song}
\affil{Department of Statistics, Kyungpook National University}
\author[2]{Jiwon Kang}
\affil{Department of Computer Science and Statistics, Jeju National University}
\date{}

\begin{document}
\maketitle

\begin{abstract}
This study considers the problem of testing for parameter change, particularly in the presence of outliers. To lessen the impact of outliers, we
propose a robust test based on the density power divergence introduced by Basu et al. (Biometrika, 1998), and then  derive its limiting null distribution. Our test procedure can be naturally extended to any parametric model to which MDPDE can be applied. To illustrate this, we apply our test procedure to GARCH models.  We demonstrate the validity and robustness of the proposed test through a simulation study. In a real data application to the Hang Seng index, our test locates some change-points that are not detected by the existing tests such as the score test and the residual-based CUSUM test.
\end{abstract}

\section{Introduction}

It is often observed, for example, that financial markets fluctuate widely by economic and political events, and it is well known that such events can cause deviating observations in data or structural breaks in underlying models. Over the past decades, most of works have dealt with these phenomena separately. For the former, researchers have developed various robust methods for reducing the impact of outlying observations. For an overview of related theories and methods, see, for example, Marona et al. (2006). 
 The latter has also been extensively studied in the field of change point analysis and vast amount of literature have been devoted to this area.
   See the recent review papers by Aue and Horváth (2013) and Horváth and Rice (2014). However, there have been relatively few studies addressing the cases that both situations are involved.

This paper is concerned with the problem of testing for parameter change, particularly in the presence of outliers.
As is well known, classical estimators such as MLE are very sensitive to outliers. Since various test statistics are constructed based on such estimators, one may naturally surmise that existing tests are also likely to be affected by outliers. 
In the literature, Tsay (1988) investigated a procedure for detecting outliers, level shifts, and variance change in a univariate time series and Lee and Na (2005) and Kang and Song (2015) introduced a estimates-based CUSUM test using a robust estimator. Recently, Fearnhead and Rigaill (2018) proposed the penalized cost function to detect the changes in the location parameter and Song (2020) proposed trimmed residual based CUSUM test for diffusion processes. These studies consistently addressed that the previous parameter change tests are also severely damaged by outliers, which obviously indicates that it is not easy to determine whether the testing results are due to genuine changes or not when outlying observations are included in a data set being suspected of having parameter changes.

In this study, we propose a robust test for parameter change using a divergence based method. Divergences are usually taken to evaluate the discrepancy between two probability distributions, but some of them have been popularly used as a way to construct robust estimators. See, for example, Basu et al. (1998), Fujisawa and Eguchi (2008), and Ghosh and Basu (2017) for density power (DP), $\gamma$-, and S-divergence based estimation methods, respectively. In this study, we employ DP divergence (DPD) to construct a robust test.
Since Basu et al. (1998) introduced the DPD-based estimation method that yields the so-called minimum DPD estimator (MDPDE), the estimation method has been successfully applied to various parametric models. See,for example, Lee and Song (2009), Kang and Lee (2014), and  Song (2017).
These studies showed that the corresponding MDPDEs have a strong robust property with little loss in efficiency. Recently, the DPD based method has been extended to testing problems.  Basu et al. (2013, 2016) used the objective function of MDPDE to propose Wald-type tests and Ghosh et al. (2016) investigated its properties.
Like the MDPDE, the induced tests are found to inherit the robust and efficient properties, and such results motivate us to consider a robust test based on DPD approach. Meanwhile, it is noteworthy that tests based on other divergences such as $\phi$- and S-divergences have also been studied before by several authors. See, for example, Batsidis et al.(2013) and  Ghosh et al. (2015). For  statistical inference based on divergences, we refer the reader to Pardo (2006).

Our robust test is constructed generalizing the score test for parameter change.  More specifically, the test in this paper is obtained by replacing the score function in the score test with the derivatives of the objective function of MDPDE. Since the score function is actually induced from Kullback-Leibler (KL) divergence, our  test can be considered as  a DP divergence  version of the score test. Noting that the DP divergence includes KL divergence, the proposed test is expected to enjoy the merits of the score test as well as robust and efficient properties. For instance, according to the previous studies such as Song and Kang (2018), the score test has a merit in that it produces stable sizes especially when true parameter lies near the boundary of parameter space. Furthermore, just as the score test can be applied to general parametric model, our test procedure is also applicable to any parametric model to which MDPDE can be applied. To demonstrate this, we first introduce the way of constructing a DPD-based test in i.i.d. cases and then apply our test procedure to GARCH models.

This paper is organized as follows. In Section 2, we propose a DPD-based test for parameter change and derive its asymptotic null distribution. In Section 3, we extend our method to GARCH models. We examine our method numerically through Monte Carlo simulations in Section 4. Section 5 illustrates a real data application and Section 6 concludes the paper. The technical proofs are provided in Appendix.

\section{DP divergence based test for parameter change}\label{Sec:2}

\setcounter{equation}{0}
In this section, we review the DP divergence and MDPDE by Basu et al. (1998) and then introduce a robust test statistic based on the divergence.

For two density functions $f$ and $g$, DP divergence is defined by
\begin{eqnarray*}\label{DPD}
d_\alpha (g, f):=\left\{\begin{array}{lc}
\displaystyle\int\Big\{f^{1+\alpha}(z)-(1+\frac{1}{\alpha})\,g(z)\,f^\alpha(z)+\frac{1}{\alpha}\,g^{1+\alpha}(z)\Big\} dz &,\alpha>0, \vspace{0.3cm}\\
\displaystyle\int g(z)\big\{ \log g(z)-\log f(z) \big\} dz
&,\alpha=0.
\end{array} \right.
\end{eqnarray*}
 As special cases, the divergence includes the KL divergence and $L_2$ distance when $\A=0$ and $\A=1$, respectively.
Since $d_\alpha(f,g)$ converges to $d_0(f,g)$ as $\A\rightarrow 0$, the above divergence with
$0<\alpha<1$ provides a smooth bridge between KL divergence and the $L_2$ distance.

Let $X_1, \cdots, X_n$ be a random sample from an unknown density $g$. To define an estimator using the divergence, consider a family of parametric densities $\{ f_{\theta} | \theta \in \Theta \subset \mathbb{R}^d \}$. Then, the
MDPDE with respect to the parametric family  $\{ f_{\theta} \}$ is defined as the estimator that minimizes the empirical version of the divergence $d_\A(g,f_\T)$. That is,
\begin{eqnarray}\label{MDPDE}
\hat \theta_{\alpha, n} = \argmin_{\theta \in \Theta}\, \frac{1}{n} \sum_{i=1}^n l_{\alpha}(X_i;\T):=\argmin_{\theta \in \Theta} H_{\alpha,n}(\theta),
\end{eqnarray}
where \begin{eqnarray*}
l_\alpha(X_i;\theta) = \left\{ \begin{array}{ll}
   \displaystyle  \int f_\theta^{1+\alpha}(z) dz - \left( 1 + \frac{1}{\alpha} \right)
     f_\theta^{\alpha}(X_i)    & \mbox{, $\alpha > 0$,}\vspace{0.15cm}\\
   \displaystyle  - \log f_\theta(X_i)      & \mbox{, $\alpha = 0$.}
   \end{array}
 \right.
\end{eqnarray*}
Here, the tuning parameter $\alpha$ plays an important role in controlling the trade-off between robustness and asymptotic efficiency of the estimator. Basu et al. (1998) showed that $\hat \theta_{\alpha,n}$ is weakly consistent for $\T_\A := \argmin_{\theta \in \Theta}\  d_{\alpha} (g, f_{\theta})$  and asymptotically normal, and demonstrated that the estimators with small $\A$ have strong robust properties with little loss in asymptotic efficiency relative to MLE.

In order to focus on the parameter change problem, we assume hereafter that $g$ belongs to the parametric family  $\{ f_{\theta} \}$, that is, $g=f_{\T_0}$ for some $\T_0 \in \Theta$. In this case, $\T_\A$ becomes equal to the true parameter vector $\T_0$. From now, for notational convenience, we use $\pa_\T$ and $\paa$ to denote $\frac{\pa}{\pa\T}$ and $\frac{\pa^2}{\pa\T\pa\T'}$, respectively.

We now intend to test the following hypotheses in the presence of outliers:
\begin{eqnarray*}\label{H}
H_0: X_1, \cdots, X_n \sim i.i.d.\ f_{\T_0}\qquad v.s.\qquad H_1:\textrm{not~} H_0.
\end{eqnarray*}
For this task, we construct a test statistics using the derivative of the objective function in (\ref{MDPDE}). Our background idea coincides with that of the score test by Horv\'{a}th and Parzen (1994). Horv\'{a}th and Parzen (1994) showed that under $H_0$,
\[\frac{1}{\sqrt{n}}\sum_{i=1}^{[ns]} \pa_\T \log f_{\hat{\T}_n}(X_i)= -\frac{[ns]}{\sqrt{n}}\pa_\T H_{\A=0,[ns]}(\hat{\T}_n) \stackrel{w}{\longrightarrow}\, J^{-1/2} B^o_d(s)\quad
\rm{in}\ \ \mathbb{D}\,\big( [0,1],\, \mathbb{R}^d\big)\,,\]
where $\hat{\T}_n$ and $J$ denote the MLE and the Fisher information matrix, respectively, and $\{B^o_d(s) |s\geq0\}$ is a $d$-dimensional standard Brownian bridge, and then used the above to propose the score test for parameter change. In this study, we extend their result to the case of $\A>0$.

By using Taylor's theorem, we have that  for each $ s \in [0,1]$,
\begin{eqnarray}\label{paH}
\frac{[ns]}{\sqrt{n}}\pa_{\theta}{H}_{\A,[ns]}(\HT)=\frac{[ns]}{\sqrt{n}}\pa_{\theta}{H}_{\A,[ns]}(\T_0)+\frac{[ns]}{n}\paa {H}_{\A,[ns]}(\T^*_{\A,n,s}) \sqrt{n}(\HT-\theta_0)
\end{eqnarray}
where $\T^*_{\A,n,s}$ is an intermediate point between $\theta_0$ and $\HT$.
Since $\pa_{\theta}{H}_{\A,n}(\HT)=0$, we have that for $s=1$,
\begin{eqnarray*}
\sqrt{n}\,\pa_{\T}{H}_{\A,n}(\T_0)+\paa {H}_{\A,n}(\T^*_{\A,n,1}) \sqrt{n}(\HT-\T_0)=0,
\end{eqnarray*}
and thus we can express that
\begin{eqnarray}\label{hat.theta}
\sqrt{n}(\HT-\T_0)= J_\A^{-1}\sqrt{n}\,\pa_{\T}{H}_{\A,n}(\T_0)+J_\A^{-1}(B_{\A,n}+J_\A)\sqrt{n}(\HT-\T_0),
\end{eqnarray}
where $B_{\A,n}=\paa{H}_{\A,n}(\T^*_{\A,n,1})$ and $J_\A$ is the one defined in the assumption {\bf A6} below. Here, putting the above into (\ref{paH}), we obtain
\begin{eqnarray}\label{paH_final}
\frac{[ns]}{\sqrt{n}}\pa_{\theta}{H}_{\A,[ns]}(\HT)
&=&\frac{[ns]}{\sqrt{n}}\pa_{\theta}{H}_{\A,[ns]}(\T_0)+\frac{[ns]}{n}\paa{H}_{\A,[ns]}(\T^*_{\A,n,s})
J_\A^{-1}\sqrt{n}\,\pa_{\T}{H}_{\A,n}(\T_0)\nonumber \\
&&+\frac{[ns]}{n}\paa {H}_{\A,[ns]}(\T^*_{\A,n,s})
J_\A^{-1}(B_{\A,n}+J_\A)\sqrt{n}(\HT-\T_0).
\end{eqnarray}
To derive the limiting null distribution of the above, the strong consistency of $\hat{\T}_{\A,n}$ is required. For this, we assume the following conditions to ensure the strong uniform convergence of the objective function $H_{\A,n}(\T)$:
\begin{enumerate}
\item[\bf A1.] The parameter space $\Theta$ is compact.
\item[\bf A2.] The density $f_\theta$ and  the integral $\int f_\T^{1+\A}(z)dz$ are continuous in $\theta$.
\item[\bf A3.] There exists a function $B(x)$ such that $|l_\A(x;\T)| \leq B(x)$ for all $x$ and $\T$ and $\E[B(X)]<\infty$.
\end{enumerate}
By the assumption {\bf A2}, $l_\A(x;\T)$ becomes a continuous function in $\T$. Hence, it follows that
\[\sup_{\T \in \Theta} \bigg|\frac{1}{n}\sum_{i=1}^n l_\A(X_i;\T) -\E[l_\A(X;\T)] \bigg|\stackrel{a.s.}{\longrightarrow}0\]
(cf. chapter 16 in Furgeson (1996)). Noting the fact that $\E[l_\A (X;\T)] =d_\A (f_{\T_0},f_\T)-\frac{1}{\A}\int f_{\T_0}(z)\ud z$, one can see that  $\E[l_\A (X;\T)]$ has the minimum value at $\T_0$. Hence, $\hat\theta_{\A,n}$ converges almost surely to $\theta_0$ by the standard arguments.
Assumption {\bf A3} is ensured by such condition that $\sup_{x,\T \in \Theta} f_\T(x) <\infty$. This condition can usually be obtained by restricting the range of
scale parameter. For example, when the normal parametric family $\{N(\mu,\sigma^2)|\mu\in\mathbb{R}, \sigma>0\}$ is considered, the condition is obtained by considering the parameter space $\Theta=\{(\mu,\sigma)|\ \sigma\geq c\}$ for some $c>0$. Another set of conditions for the strong consistency can be found, for example, in Lee and Na (2005). We introduce further assumptions. Throughout this paper, the symbol $\|\cdot\|$ denotes any norm for matrices and vectors.
\begin{enumerate}
\item[\bf A4.] The integral $\int f_\T^{1+\A}(z)dz$ is twice differentiable with respect to $\T$ and the derivative can be taken under the integral sign.
\item[\bf A5.] $\paa l_\A (x;\T)$ is continuous in $\T$ and there exists an open neighborhood $N(\T_0)$ of $\T_0$ such that $\E[ \sup_{\T \in N(\T_0)} \big\| \paa l_\A (X;\T)\big\|] <\infty$.
\item[\bf A6.] The matrices $K_\A$ and  $J_\A$ defined by
\begin{eqnarray*}
K_\A = \E\big[ \pa_\T l_\A(X;\T_0) \pa_{\T'} l_\A(X;\T_0)\big]\quad\mbox{and}\quad
J_\A=-\E \big[ \paa l_\A(X;\T_0)\big]
\end{eqnarray*}
exist and are non-singular.
\end{enumerate}
The following is the first main result in this study, which is used as the building block for constructing a robust test.
\begin{theorem}\label{Thm1}
Suppose that the assumptions {\bf A1}- {\bf A6} hold. Then, under $H_0$, we
have that for $\alpha\geq0$,
\begin{eqnarray*}
K_\A^{-1/2}\frac{[ns]}{\sqrt{n}}\pa_{\theta}H_{\A,[ns]}(\HT)\stackrel{w}{\longrightarrow}\,B^o_d(s)\quad
\rm{in}\ \ \mathbb{D}\,\big( [0,1],\, \mathbb{R}^d\big)\,,
\end{eqnarray*}
where $\{B^o_d(s)|s\geq0\}$ is a $d$-dimensional standard Brownian bridge.
\end{theorem}
Using Theorem \ref{Thm1}, one can construct a DPD based test for parameter constancy as follows.
\begin{theorem}\label{Thm2}
Suppose that the assumptions {\bf A1}- {\bf A6} hold. Then, under $H_0$, we
have that for $\alpha\geq0$,
\begin{eqnarray*}
T_n^\alpha:=\max_{1\leq k \leq n}
\frac{k^2}{n}\pa_{\theta'}H_{\A,k}(\HT)K_\A^{-1}
\pa_{\theta}H_{\A,k}(\HT)\stackrel{d}{\longrightarrow}
\sup_{0\leq s \leq 1} \big\|B_d^o(s)\big\|_2^2.
\end{eqnarray*}
\end{theorem}
\noindent
To implement the test above, it needs to replace $K_\A$ with a consistent estimates. As a natural estimator of $K_\A$, one can consider to use
\[\hat{K}_\A = \frac{1}{n}\sum_{i=1}^n \pa_\T l_\A(X_i;\hat\T_{\A,n})\pa_{\T'} l_\A(X_i;\hat\T_{\A,n}).\]
Under the condition that $\E\sup_{\T \in \Theta^*} \big\| \pa_\T l_\A(X;\T)\pa_{\T'} l_\A(X;\T)\big\|<\infty$ for some neighborhood $\Theta^* \subset \Theta$ of $\T_0$, it can be shown that $\hat{K}_\A$ converges to $K_\A$ in probability.

\begin{remark}
Since $-H_{0,n}(\theta)$ is the log likelihood, $T_n^\A$ with $\A=0$ becomes the score test presented by Horv\'{a}th and Parzen (1994)
\end{remark}

\begin{remark}Noting that $\pa_\theta  H_{\alpha,n}(\hat\theta_{\A,n})=0$, it can be written that
\begin{eqnarray*}
\frac{[ns]}{\sqrt{n}}\pa_\theta H_{\A,[ns]}(\hat\theta_{\A,n})&=&\frac{1}{\sqrt{n}}\bigg\{ \sum_{i=1}^{[ns]} \pa_\theta l_\A (X_i;\hat\theta_{\A,n})-\frac{[ns]}{n} \sum_{i=1}^{n} \pa_\theta  l_\A(X_i ;\hat\theta_{\A,n})\bigg\} \\
&=&\frac{[ns]}{n} \Big(1-\frac{[ns]}{n}\Big) \sqrt{n} \Big(\frac{1}{[ns]}\sum_{i=1}^{[ns]} \pa_\theta l_\A(X_i ;\hat\theta_{\A,n}) -\frac{1}{n-[ns]}\sum_{i=[ns]+1}^n \pa_\theta l_\A(X_i; \hat\theta_{\A,n})\Big).
\label{locate}
\end{eqnarray*}
Our test can therefore be regarded as a CUSUM-type test based on $\{\pa_\T l_\alpha (X_i; \theta)\}$.  When $H_0$ is rejected by such CUSUM-type test,
the change-point is located as the argument that maximizes the absolute value of the cumulative sum in test statistics. See, for example, Robbins et al. (2011). For the same reason, the change-point estimator of the test above is obtained as
\[ \hat k:=\argmax_{1\leq k \leq n}\,  \frac{k^2}{n}\pa_{\theta'}H_{\A,k}(\hat{\theta}_{\A,n}){\hat K_\A}^{-1}
\pa_{\theta}H_\A(\hat{\theta}_{\A,n}).\]
\end{remark}

\begin{remark}\label{ALPHA}
Selection of the optimal $\A$ is an important practical issue. Several authors studied decision criteria to choose an optimal $\A$. For example,  Warwick (2005) proposed a selection rule for $\A$ that minimizes the asymptotic mean squared error, Fujisawa and Eguchi (2006)
introduced an adaptive method based on the Cramer–von Mises divergence, and  Durio and Isaia (2011) considered a bootstrap method based on the similarity measure between MDPD estimate and ML estimate. It should, however, be noted that the existing studies dealt with the problem in estimation situation, that is, under the assumption that there exists no parameter change.
In testing procedure, the selection of $\A$ is more complicated. If $H_0$ is not rejected  by the proposed test $T_n^\A$ with all $\A$ considered, one may employ the existing decision criteria aforementioned. However, in the cases where $H_0$ is rejected, indeed, it seems difficult to establish a decision rule. According to the simulation study below, the empirical power of $T_n^\A$ shows a tendency to decrease  with an increase in $\A$. In particular, $T_n^\A$ with small $\A$ produces powers almost similar to that of the score test when the data is uncontaminated, while keeping strong robustness. This indicates that small $\A$ may be preferred because a large $\A$ can lead to a significant loss in powers when the degree of contamination is not as large as speculated. Based on our simulation results, we recommend to use an $\A$ in [0.1,0.3] when practitioners do not find a proper decision rule.
\end{remark}
\begin{remark}
Although, to the best of our knowledge, there are not yet in-depth studies on systematic selection in testing problem, one may consider to choose an $\A$ in terms of forecasting performance. To this end,  for each $\A$ under consideration, conduct $T_n^\A$ to detect change-points. Then, using the data from the last change-point, estimate the model and calculate forecasting error measures such as root mean squared errors (cf. Song (2020)). Based on the obtained values, one can select a proper $\A$. In our data analysis, we illustrate the procedure to calculate forecasting errors using the model induced by each $T_n^\A$.
\end{remark}

\begin{remark} \label{BS}The binary segmentation procedure can be used to find multiple changes as do other CUSUM-type tests. That is, first, (i)
perform the test $T_n^\A$ on the whole series $\{X_1,\cdots,X_n\}$. If $H_0$ is rejected, split the series into two
subseries $\{X_t,\cdots, X_{\hat k}\}$ and $\{X_{\hat k +1},\cdots, X_n\}$, where  $\hat k$ is the one in Remark 2. Then, (ii) repeating the same
procedure on each subseries until no change-point is detected, one can locate multiple change-points. For
more details on the binary segmentation procedure of CUSUM-type test, see Aue and Horv\'{a}th(2013) and references therein.
\end{remark}

As studied in several papers stated in Introduction, the MDPD estimation procedure can be conveniently applied to various parametric models including time series models and multivariate models.
Once such MDPDE is set up, the robust test procedure can be extended to corresponding models. As an illustration, we address a DPD based test for GARCH models using the MDPDE established in Lee and Song (2009). All the remarks mentioned above still hold for the extended cases.

\section{DP divergence based test for GARCH models}\label{Sec:3}

Consider the following GARCH($p,q$) model:
\begin{eqnarray*}\label{GARCH}
\begin{split}
X_t &= \sigma_t\,\epsilon_t,\\
\sigma_t^2&:=\sigma_t^2(\T)= \omega+\sum_{i=1}^p\A_{i} X_{t-i}^2+\sum_{j=1}^q\B_{j}\sigma^2_{t-j},
\end{split}
\end{eqnarray*}
where $\T=(\W,\A_1,\cdots,\A_p,\B_1,\cdots,\B_q)'\in \Theta$ in $(0,\infty)\times[0,\infty)^{p+q}$ and $\{\epsilon_t | t\in
\mathbb{Z}\}$ is a sequence of i.i.d. random variables with zero mean and unit variance.
We assume that the process $\{X_t| t \in \mathbb{Z}\}$ is strictly stationary and ergodic.
The conditions for the existence of stationary and ergodic process can be found, for example, in  Bougerol and Picard (1992).

In order to estimate the unknown parameter in the presence of outliers,
Lee and Song (2009) introduced MDPDE for the GARCH model as follows:
\begin{eqnarray}\label{root}
 \hat{\theta}_{\A,n}= \argmin_{\theta \in \Theta}\frac{1}{n}\sum_{t=1}^n \tilde{l}_\A(X_t;\T)
 := \argmin_{\theta \in \Theta}\tilde{H}_{\A,n}(\T),
\end{eqnarray}
 where
\begin{eqnarray*}
\tilde{l}_\A(X_t;\T) &=&
\left\{
\begin{array}{lc}{\displaystyle\Big(\frac{ 1}{\sqrt{{\tilde{\sigma}}_t^2}}\Big)^\alpha
\Big\{\frac{1}{\sqrt{1+\alpha}}-\Big(1+\frac{1}{\alpha}\Big)
\exp\Big(-\frac{\A}{2}\frac{ X_t^2}{{\tilde{\sigma}}_t^2 }\Big)
\Big\}} &,\A > 0
 \\ \\
 {\displaystyle \frac{X_t^2}{\tilde{\sigma}_t^2}
+\log\tilde{\sigma}_t^2}&,\A=0\end{array}\right.
\end{eqnarray*}
and $\{{\tilde{\sigma}}_t^2| 1\leq t\leq n\}$ is given recursively by
\begin{eqnarray}\label{tilde1}
{\tilde{\sigma}}_t^2:={\tilde{\sigma}}_t^2(\T)= \W + \sum_{i=1}^{p}\alpha_i\, X_{t-i}^2+\sum_{j=1}^{q}\beta_j\,
{\tilde{\sigma}}_{t-j}^2.
\end{eqnarray}
Here, the initial values could be any constant values taken to be fixed, neither random nor a function of the parameters. So as to obtain the asymptotic properties of the MDPDE, the following regularity conditions are imposed.

\begin{enumerate}
\item[\bf A1.] The true parameter vector $\theta_0 \in \Theta\,$ and $\Theta$ is compact.
\item[\bf A2.] $\displaystyle \sup_{\T\in\Theta}\sum_{j=1}^{q}\beta_j <1.$
\item[\bf A3.] If $q>0\,$, ${\mathcal{A}}_{\theta_0}(z)$ and
${\mathcal{B}}_{\theta_0}(z)$ have no common root,
${\mathcal{A}}_{\theta_0}(1) \neq 1$, and $\alpha_{0p}+\beta_{0q}
\neq 0$, where ${\mathcal{A}}_{\T}(z)=\sum_{i=1}^p \A_i\,z^i$
and ${\mathcal{B}}_{\T}(z)=1-\sum_{j=1}^q\beta_j\,z^j$.
(Conventionally, ${\mathcal{A}}_{\T}(z)=0$ if $p=0$ and
${\mathcal{B}}_{\T}(z)=1$ if $q=0$.)
\item[\bf A4.] $\theta_0$ is in the interior of $\Theta$.
\end{enumerate}
 The following asymptotics of the MDPDE  are established by Lee and Song (2009).

\begin{Prop}\label{Prop1}
For each $\A \geq 0,$ let $\{ \hat{\T}_{\A,n}\}$ be a sequence of
the MDPDEs satisfying $\eqref{root}$. Suppose that $\epsilon_t$s are i.i.d. random variables from $N(0,1)$.
Then, under the assumptions $\bf A1$-$\bf A3$, $\hat{\T}_{\A,n}$ converges
to $\theta_0$ almost surely. If, in addition, the assumption {\bf A4} holds, then
\[\sqrt{n}\,(\,\hat{\theta}_{\A,n} -\,\theta_0) \stackrel
{d}{\longrightarrow} N\,\Big(\,0\,,\frac{k(\alpha)}{g^2(\alpha)}\,
J_{2,\alpha}^{-1} \,J_{1,\alpha} \,J_{2, \alpha}^{-1}\,\Big)\,,\]
where
\begin{eqnarray*}
k(\alpha)=\frac{{(1+\alpha)}^2(1+2\alpha^2)}{2{(1+2\alpha)}^{2/5}}-\frac{\alpha^2}{4(1+\alpha)}\,,\quad
g(\alpha)=\frac{\alpha^2+2\A+2}{4(1+\alpha)^{3/2}}\,,
\end{eqnarray*}
and
\begin{eqnarray*}
J_{1,\alpha}=\E \left[{\Big(\frac{1}{\sigma_t^2(\theta_0)}\Big)}^{\alpha+2}\,\pa_\T \sigma_t^2(\theta_0) \pa_{\T'}
\sigma_t^2(\theta_0)  \right]\,,&&
J_{2,\alpha}=\E \left[
{\Big(\frac{1}{\sigma_t^2(\theta_0)}\Big)}^{\frac{\alpha}{2}+2}\,\pa_\T
\sigma_t^2(\theta_0) \pa_{\T'}
\sigma_t^2(\theta_0)  \right].
\end{eqnarray*}
\end{Prop}
\begin{remark}\label{Rm.l}
Using (\ref{pa.l}) in Appedix below, one can see that
\[k(\A)J_{1,\A}=\E\big[ \pa_\T l_\A(X_t;\T_0) \pa_{\T'} l_\A(X_t;\T_0)\big]\quad\mbox{and}\quad
g(\A)J_{2,\A}=\E \big[ \paa l_\A(X_t;\T_0)\big],\]
where $l_\A(X_t;\T)$ is the counterpart of $\tilde l_\A(X_t;\T)$ obtained by replacing $\tilde \sigma_t^2$ with  $\sigma_t^2$ in (\ref{root}). The non-singularity of $J_{2,\A}$ is shown in page 337 in Lee and Song (2009), and one can also show the invertibility of $J_{1,\A}$  in a similar fashion.
\end{remark}
Now, we construct DPD based test for the following hypotheses:
\begin{eqnarray*}
H_0: \theta_0\textrm{ does not change over~} X_1,\ldots,X_n\qquad v.s.\qquad H_1:\textrm{not~} H_0.
\end{eqnarray*}
Using Taylor's theorem with the same arguments used to obtain (\ref{paH_final}), we can have that  for each $ s \in [0,1]$,
\begin{eqnarray}\label{paH_final2}
\frac{[ns]}{\sqrt{n}}\pa_{\theta}\tilde{H}_{\A,[ns]}(\HT)
&=&\frac{[ns]}{\sqrt{n}}\pa_{\theta}\tilde{H}_{\A,[ns]}(\T_0)+\frac{[ns]}{n}\paa\tilde{H}_{\A,[ns]}(\T^*_{\A,n,s})
J_\A^{-1} \sqrt{n}\,\pa_{\T}\tilde{H}_{\A,n}(\T_0)\nonumber\\
&&+\frac{[ns]}{n}\paa \tilde{H}_{\A,[ns]}(\T^*_{\A,n,s})
J_\A^{-1}(B_{\A,n}+J_\A)\sqrt{n}(\HT-\T_0),
\end{eqnarray}
where $\T^*_{\A,n,s}$ is an intermediate point between $\theta_0$ and $\HT$, $B_{\A,n}=\paa\tilde{H}_{\A,n}(\T^*_{\A,n,1})$, and $J_\A=-g(\A)J_{2,\A}$.
In Appendix below, we show that the last term in the RHS of the above equation is asymptotically negligible and the first two term converges weakly to Brownian bridge, where Lemmas \ref{Lm.G1} and \ref{Lm.G4} play a critical role. The following theorem states that the similar result to Theorem \ref{Thm1} is also obtained in GARCH models.
\begin{theorem}\label{Thm3}
Suppose that the assumptions {\bf A1}- {\bf A4} hold. Then, under $H_0$, we
have
\begin{eqnarray*}
\frac{1}{\sqrt{k(\A)}}J_{1,\A}^{-1/2}\frac{[ns]}{\sqrt{n}}\pa_{\theta}\tilde{H}_{\A,[ns]}(\HT)\stackrel{w}{\longrightarrow}\,B^o_D(s)\quad
\rm{in}\ \ \mathbb{D}\,\big( [0,1],\, \mathbb{R}^D\big)\,,
\end{eqnarray*}
where $D=p+q+1$ and $\{B^o_D(s)\}$ is a $D$-dimensional standard Brownian bridge,
and thus,
\begin{eqnarray*}
\tilde T_n^\alpha:=\frac{1}{k(\A)}\max_{1\leq k \leq n}
\frac{k^2}{n}\pa_{\T'}\tilde{H}_{\A,k}(\HT) J_{1,\A}^{-1}\,
\pa_{\theta}\tilde{H}_{\A,k}(\HT)\stackrel{d}{\longrightarrow}
\sup_{0\leq s \leq 1} \big\|B_D^o(s)\big\|_2^2.
\end{eqnarray*}
\end{theorem}
\noindent  Recalling that $k(\A)J_{1,\A}=\E\big[\pa_\T l_\A(X;\T_0)\pa_{\T'} l_\A(X;\T_0)\big] $, one can estimate  $J_{1,\A}$ as follows:
\[ \hat J_{1,\A}=\frac{1}{k(\A)n} \sum_{t=1}^n \pa_\T \tilde l_\A(X_t;\hat\T_{\A,n})\,\pa_{\T'} \tilde l_\A(X_t;\hat\T_{\A,n}).\]
The consistency of  $\hat J_{1,\A}$ is proved in Lemma \ref{Lm.G6}.
\begin{remark}
Berkes et al. (2004) proposed a score test for parameter change in GARCH models. Although their test is constructed using the quasi-MLE of Berkes and Horv\'{a}th (2004), the test is essentially equal to $\tilde T_n^\A$ with $\A=0$.
\end{remark}
\section{Simulation results}
In the present section, we evaluate the finite sample performance of the proposed test and compare with the score test. All empirical sizes and powers in this section are calculated  at 5\% significance level  based on 2,000 repetitions. The corresponding critical values are obtained via Monte Carlo simulations.
\begin{table}[!b]
  \caption{\noindent \small Empirical sizes and powers of $T_n^0$ and $T_n^\alpha$ with and without outliers.}\label{tab:iid1}
   \centering
  \tabcolsep=2.3pt
 {\footnotesize
       \begin{tabular}{lcrrrrrrclccccccc}
\cmidrule{1-8}\cmidrule{10-17}    \multirow{2}[4]{*}{No outliers} &       &       & \multicolumn{5}{c}{$T_n^\alpha$}      &       & \multirow{2}[4]{*}{$p=1\%, \delta=10$} &       &       & \multicolumn{5}{c}{$T_n^\alpha$} \\
\cmidrule{4-8}\cmidrule{13-17}          & $n$   & \multicolumn{1}{l}{$T_n^0$} & \multicolumn{1}{c}{} & \multicolumn{1}{c}{0.1} & \multicolumn{1}{c}{0.2} & \multicolumn{1}{c}{0.3} & \multicolumn{1}{c}{0.5} &       &       & $n$   & \multicolumn{1}{l}{$T_n^0$} &       & 0.1   & 0.2   & 0.3   & 0.5 \\
\cmidrule{1-8}\cmidrule{10-17}    Sizes & 500   & 0.034 &       & 0.040 & 0.044 & 0.042 & 0.043 &       & Sizes & 500   & 0.026 &       & 0.047 & 0.046 & 0.048 & 0.046 \\
    $(\mu,\sigma^2)=(0,1)$ & 1000  & 0.041 &       & 0.046 & 0.047 & 0.048 & 0.052 &       & $(\mu,\sigma^2)=(0,1)$ & 1000  & 0.030 &       & 0.043 & 0.042 & 0.046 & 0.044 \\
\cmidrule{1-8}\cmidrule{10-17}    $\mu:0\rightarrow0.15$ & 500   & 0.221 &       & 0.222 & 0.218 & 0.214 & 0.198 &       & $\mu:0\rightarrow0.15$ & 500   & 0.118 &       & 0.216 & 0.207 & 0.201 & 0.186 \\
          & 1000  & 0.469 &       & 0.466 & 0.460 & 0.436 & 0.406 &       &       & 1000  & 0.250 &       & 0.464 & 0.452 & 0.436 & 0.399 \\
\cmidrule{1-8}\cmidrule{10-17}    $\mu:0\rightarrow0.3\ $ & 500   & 0.764 &       & 0.758 & 0.752 & 0.735 & 0.696 &       & $\mu:0\rightarrow0.3\ $ & 500   & 0.528 &       & 0.754 & 0.740 & 0.722 & 0.680 \\
          & 1000  & 0.984 &       & 0.980 & 0.976 & 0.972 & 0.956 &       &       & 1000  & 0.834 &       & 0.978 & 0.978 & 0.972 & 0.955 \\
\cmidrule{1-8}\cmidrule{10-17}    $\sigma^2:1\rightarrow 1.25$ & 500   & 0.247 &       & 0.246 & 0.232 & 0.216 & 0.193 &       & $\sigma^2:1\rightarrow 1.25$ & 500   & 0.026 &       & 0.230 & 0.230 & 0.222 & 0.196 \\
          & 1000  & 0.503 &       & 0.503 & 0.480 & 0.452 & 0.387 &       &       & 1000  & 0.038 &       & 0.466 & 0.450 & 0.420 & 0.376 \\
\cmidrule{1-8}\cmidrule{10-17}    $\sigma^2:1\rightarrow 1.5$ & 500   & 0.704 &       & 0.702 & 0.670 & 0.638 & 0.567 &       & $\sigma^2:1\rightarrow 1.5$ & 500   & 0.040 &       & 0.698 & 0.688 & 0.654 & 0.574 \\
          & 1000  & 0.968 &       & 0.966 & 0.958 & 0.946 & 0.900 &       &       & 1000  & 0.068 &       & 0.958 & 0.950 & 0.935 & 0.883 \\
\cmidrule{1-8}\cmidrule{10-17}    \end{tabular}}
\end{table}%

We first consider i.i.d. cases to see the behaviors of the tests in the presence of outliers.  For this, we generate contaminated samples $\{X_t\}$ by using the following scheme: $X_t=X_{t,o} + \delta\ p_t\cdot  sign(X_{t,o})$, where  $\{X_{t,o}\}$ is a sequence of i.i.d. random variables from $N(\mu,\sigma^2)$, $\delta$ is a positive constant, and $p_t$s are i.i.d. Bernoulli random variables with success probability $p$. $\{X_{t,o}\}$ and $\{p_t\}$ are assumed to be independent. This setting describes the situation that the original data $\{X_{t,o}\}$ is contaminated by outlier process $\{\delta p_t\}$. Uncontaminated samples are obtained with $p=0$ or $\delta=0$. $(\mu,\sigma^2)=(0,1)$ is considered to evaluate empirical sizes, and we change the parameter $(\mu,\sigma^2)$ at  midpoint $t=n/2$ for empirical powers.
 The empirical sizes and powers are presented in Table \ref{tab:iid1}, where the left sub-table shows the results for uncontaminated case and the right for contaminated case with $p=1\%$ and $\delta=10$. In the left sub-table, one can see that  $T_n^\alpha$ yields proper sizes and reasonable powers in all $\alpha$ considered, and the score test $T_n^0$ shows best performance as expected. It is noteworthy that the power tends to decrease as $\alpha$ increases and that $T_n^\alpha$ with $\alpha$ close to 0 shows similar performance to $T_n^0$. 
 In the right sub-table, we can observe the power losses of $T_n^0$. In particular, $T_n^0$ is severely compromised in testing for the change in $\sigma^2$, that is, variance change. In contrast, $T_n^{\alpha}$ produces empirical powers similar to the powers obtained in the left sub-table, i.e., uncontaminated cases. This indicates that $T_n^{\alpha}$ is less affected by outliers.
 Such power losses of the score test and the robustness of the proposed test are clearly shown in Table \ref{tab:iid2}, which present the results for more contaminated case. In all contaminated cases, size distortions are not observed.

\begin{table}[h]
\centering
\caption{ \small Empirical sizes and powers of  $T_n^0$ and $T_n^\alpha$ under more severe contamination.}\label{tab:iid2}
  \tabcolsep=2.3pt
   \renewcommand{\arraystretch}{1}
    {\footnotesize
       \begin{tabular}{lcccccccclccccccc}
\cmidrule{1-8}\cmidrule{10-17}    \multirow{2}[4]{*}{$p=1\%, \delta=15$} &       &       & \multicolumn{5}{c}{$T_n^\alpha$}      &       & \multirow{2}[4]{*}{$p=3\%, \delta=10$} &       &       & \multicolumn{5}{c}{$T_n^\alpha$} \\
\cmidrule{4-8}\cmidrule{13-17}          & $n$   & \multicolumn{1}{l}{$T_n^0$} &       & 0.1   & 0.2   & 0.3   & 0.5   &       &       & $n$   & \multicolumn{1}{l}{$T_n^0$} &       & 0.1   & 0.2   & 0.3   & 0.5 \\
\cmidrule{1-8}\cmidrule{10-17}    Sizes & 500   & 0.018 &       & 0.040 & 0.042 & 0.044 & 0.046 &       & Sizes & 500   & 0.034 &       & 0.049 & 0.050 & 0.050 & 0.046 \\
    $(\mu,\sigma^2)=(0,1)$ & 1000  & 0.024 &       & 0.045 & 0.046 & 0.046 & 0.047 &       & $(\mu,\sigma^2)=(0,1)$ & 1000  & 0.034 &       & 0.055 & 0.051 & 0.049 & 0.046 \\
\cmidrule{1-8}\cmidrule{10-17}    $\mu:0\rightarrow0.15$ & 500   & 0.100 &       & 0.237 & 0.230 & 0.227 & 0.215 &       & $\mu:0\rightarrow0.15$ & 500   & 0.095 &       & 0.226 & 0.222 & 0.216 & 0.204 \\
          & 1000  & 0.172 &       & 0.440 & 0.429 & 0.418 & 0.370 &       &       & 1000  & 0.160 &       & 0.439 & 0.420 & 0.406 & 0.384 \\
\cmidrule{1-8}\cmidrule{10-17}    $\mu:0\rightarrow0.3\ $ & 500   & 0.408 &       & 0.762 & 0.750 & 0.730 & 0.692 &       & $\mu:0\rightarrow0.3\ $ & 500   & 0.308 &       & 0.771 & 0.756 & 0.737 & 0.698 \\
          & 1000  & 0.641 &       & 0.974 & 0.972 & 0.966 & 0.954 &       &       & 1000  & 0.612 &       & 0.976 & 0.972 & 0.964 & 0.946 \\
\cmidrule{1-8}\cmidrule{10-17}    $\sigma^2:1\rightarrow 1.25$ & 500   & 0.023 &       & 0.242 & 0.232 & 0.224 & 0.196 &       & $\sigma^2:1\rightarrow 1.25$ & 500   & 0.036 &       & 0.230 & 0.222 & 0.208 & 0.188 \\
          & 1000  & 0.038 &       & 0.492 & 0.460 & 0.432 & 0.368 &       &       & 1000  & 0.040 &       & 0.466 & 0.460 & 0.424 & 0.370 \\
\cmidrule{1-8}\cmidrule{10-17}    $\sigma^2:1\rightarrow 1.5$ & 500   & 0.026 &       & 0.687 & 0.670 & 0.632 & 0.563 &       & $\sigma^2:1\rightarrow 1.5$ & 500   & 0.034 &       & 0.687 & 0.692 & 0.650 & 0.574 \\
          & 1000  & 0.040 &       & 0.962 & 0.950 & 0.940 & 0.894 &       &       & 1000  & 0.049 &       & 0.948 & 0.952 & 0.936 & 0.890 \\
\cmidrule{1-8}\cmidrule{10-17}    \end{tabular}}%
\end{table}%
Next, we examine the performance of $\tilde T_n^\A$ in the following GARCH(1,1) model:
\begin{eqnarray*}
\begin{split}
X_t &= \sigma_t\,\epsilon_t,\\
\sigma_t^2&= w+\A_1 X_{t-1}^2+\B_1\sigma^2_{t-1},
\end{split}
\end{eqnarray*}
where $\{\epsilon_t \}$ is a sequence of i.i.d. random variables from $N(0,1)$.
  For empirical sizes, we generate samples with $(w,\alpha_1,\beta_1)=(0.5,0.2,0.4)$ and $(0.5,0.15,0.8)$. The latter parameter value is employed to see the performance in more volatile situation. 
  Two types of outliers, innovation outliers (IO) and  additive outliers (AO), are considered. We generate
samples with IO by replacing $\epsilon_t$ in the GARCH model above with contaminated error $\tilde \epsilon_t=\epsilon_t+ |Z_{t,c}|p_t \cdot sign(\epsilon_t)$, where $\{Z_{t,c}\}$ and $\{p_t\}$ are sequences of i.i.d. random variables from  $N(0,\sigma_c^2)$ and Bernoulli distribution with parameter $p$, respectively. It is assumed that $\{\epsilon_t\}$, $\{Z_{t,c}\}$, and $\{p_t\}$ are all independent. Samples contaminated by AO are obtained by the following model: $X_t=X_{t,o} + |Z_{t,c}|p_t\cdot sign(X_{t,o})$, where $\{X_{t,o}\}$ is the uncontaminated sample from the GARCH(1,1) model above.
 Simulation results for the cases of $(w,\alpha_1,\beta_1)=(0.5,0.2,0.4)$ and  $(0.5,0.15,0.8)$ are provided in the left and right sub-tables in Tables \ref{tab:GARCH} - \ref{GARCH:AO2}, respectively.

\begin{table}
\vspace{-0.8cm}
\centering
  \noindent\caption{\small Empirical sizes and powers of $\tilde T_n^0$ and $\tilde T_n^\alpha$ in GARCH (1,1) models without outliers.}\label{tab:GARCH}%
  \tabcolsep=2.3pt
  {\footnotesize
   \begin{tabular}{lcccccccclccccccc}
\cmidrule{1-8}\cmidrule{10-17}      &       &       & \multicolumn{5}{c}{$\tilde T_n^\alpha$}      &       &       &       &       & \multicolumn{5}{c}{$\tilde T_n^\alpha$} \\
\cmidrule{4-8}\cmidrule{13-17}      & $n$   & $\tilde T^0_n$ &       & 0.1   & 0.2   & 0.3   & 0.5   &       &       & $n$   & \multicolumn{1}{l}{$\tilde T^0_n$} &       & 0.1   & 0.2   & 0.3   & 0.5 \\
\cmidrule{1-8}\cmidrule{10-17}Sizes & 500   & 0.032 &       & 0.034 & 0.038 & 0.036 & 0.035 &       & Sizes & 500   & 0.051 &       & 0.050 & 0.054 & 0.055 & 0.052 \\
$(w, \alpha_1, \beta_1)$ & 1000  & 0.032 &       & 0.036 & 0.036 & 0.040 & 0.042 &       & $(w, \alpha_1, \beta_1)$ & 1000  & 0.040 &       & 0.038 & 0.034 & 0.037 & 0.036 \\
=(0.5, 0.2, 0.4) & 1500  & 0.038 &       & 0.038 & 0.040 & 0.043 & 0.040 &       & =(0.5, 0.15, 0.8) & 1500  & 0.044 &       & 0.043 & 0.046 & 0.044 & 0.048 \\
\cmidrule{1-8}\cmidrule{10-17}$w:0.5\rightarrow0.8$ & 500   & 0.296 &       & 0.275 & 0.254 & 0.235 & 0.184 &       & $w:0.5\rightarrow0.2$ & 500   & 0.182 &       & 0.138 & 0.104 & 0.080 & 0.070 \\
      & 1000  & 0.774 &       & 0.772 & 0.734 & 0.678 & 0.566 &       &       & 1000  & 0.378 &       & 0.318 & 0.246 & 0.186 & 0.138 \\
      & 1500  & 0.960 &       & 0.958 & 0.944 & 0.920 & 0.848 &       &       & 2000  & 0.848 &       & 0.812 & 0.743 & 0.642 & 0.466 \\
\cmidrule{1-8}\cmidrule{10-17}$\alpha_1:0.2\rightarrow0.5$ & 500   & 0.338 &       & 0.390 & 0.404 & 0.401 & 0.352 &       & $\alpha_1:0.15\rightarrow0.05$ & 500   & 0.264 &       & 0.274 & 0.274 & 0.263 & 0.244 \\
      & 1000  & 0.890 &       & 0.908 & 0.896 & 0.876 & 0.812 &       &       & 1000  & 0.729 &       & 0.760 & 0.734 & 0.694 & 0.588 \\
      & 1500  & 0.992 &       & 0.992 & 0.991 & 0.987 & 0.966 &       &       & 1500  & 0.960 &       & 0.964 & 0.949 & 0.928 & 0.852 \\
\cmidrule{1-8}\cmidrule{10-17}$\beta_1:0.4\rightarrow0.6$ & 500   & 0.337 &       & 0.330 & 0.300 & 0.268 & 0.219 &       & $\beta_1:0.8\rightarrow0.7$ & 500   & 0.222 &       & 0.204 & 0.176 & 0.152 & 0.135 \\
      & 1000  & 0.866 &       & 0.868 & 0.840 & 0.800 & 0.675 &       &       & 1000  & 0.590 &       & 0.588 & 0.554 & 0.501 & 0.412 \\
      & 1500  & 0.991 &       & 0.988 & 0.980 & 0.970 & 0.916 &       &       & 1500  & 0.892 &       & 0.898 & 0.880 & 0.832 & 0.722 \\
\cmidrule{1-8}\cmidrule{10-17}\end{tabular}}
%
  \caption{\small Empirical sizes and powers of $\tilde T_n^0$ and $\tilde T_n^\alpha$ in the case of IO contamination with $p=1\%$ and $\sigma_v^2=10$.} \label{GARCH:IO1}%
   \centering
  \tabcolsep=2.3pt
   {\footnotesize
\begin{tabular}{lcccccccrlccccccc}
\cmidrule{1-8}\cmidrule{10-17}      &       &       & \multicolumn{5}{c}{$\tilde T_n^\alpha$}      &       &       &       &       & \multicolumn{5}{c}{$\tilde T^\A_n$} \\
\cmidrule{4-8}\cmidrule{13-17}      & $n$   & $\tilde T^0_n$ &       & 0.1   & 0.2   & 0.3   & 0.5   &       &       & $n$   & $\tilde T^0_n$ &       & 0.1   & 0.2   & 0.3   & 0.5 \\
\cmidrule{1-8}\cmidrule{10-17}Sizes & 500   & 0.048 &       & 0.040 & 0.044 & 0.046 & 0.046 &       & Sizes & 500   & 0.090 &       & 0.047 & 0.043 & 0.043 & 0.042 \\
$(w, \alpha_1, \beta_1)$ & 1000  & 0.028 &       & 0.038 & 0.044 & 0.046 & 0.044 &       & $(w, \alpha_1, \beta_1)$ & 1000  & 0.054 &       & 0.047 & 0.046 & 0.046 & 0.045 \\
=(0.5, 0.2, 0.4) & 1500  & 0.028 &       & 0.042 & 0.047 & 0.049 & 0.046 &       & =(0.5, 0.15, 0.8) & 1500  & 0.041 &       & 0.052 & 0.045 & 0.045 & 0.046 \\
\cmidrule{1-8}\cmidrule{10-17}$w:0.5\rightarrow0.8$ & 500   & 0.156 &       & 0.204 & 0.267 & 0.271 & 0.230 &       & $w:0.5\rightarrow0.2$ & 500   & 0.144 &       & 0.109 & 0.086 & 0.078 & 0.068 \\
      & 1000  & 0.254 &       & 0.576 & 0.658 & 0.658 & 0.576 &       &       & 1000  & 0.230 &       & 0.260 & 0.224 & 0.195 & 0.146 \\
      & 1500  & 0.322 &       & 0.786 & 0.853 & 0.852 & 0.804 &       &       & 2000  & 0.422 &       & 0.700 & 0.664 & 0.581 & 0.424 \\
\cmidrule{1-8}\cmidrule{10-17}$\alpha_1:0.2\rightarrow0.5$ & 500   & 0.214 &       & 0.378 & 0.430 & 0.446 & 0.405 &       & $\alpha_1:0.15\rightarrow0.05$ & 500   & 0.218 &       & 0.331 & 0.380 & 0.368 & 0.348 \\
      & 1000  & 0.520 &       & 0.878 & 0.906 & 0.895 & 0.848 &       &       & 1000  & 0.394 &       & 0.806 & 0.830 & 0.816 & 0.758 \\
      & 1500  & 0.716 &       & 0.987 & 0.992 & 0.990 & 0.971 &       &       & 1500  & 0.582 &       & 0.976 & 0.980 & 0.970 & 0.943 \\
\cmidrule{1-8}\cmidrule{10-17}$\beta_1:0.4\rightarrow0.6$ & 500   & 0.216 &       & 0.317 & 0.348 & 0.317 & 0.266 &       & $\beta_1:0.8\rightarrow0.7$ & 500   & 0.184 &       & 0.218 & 0.238 & 0.231 & 0.204 \\
      & 1000  & 0.509 &       & 0.856 & 0.870 & 0.842 & 0.752 &       &       & 1000  & 0.336 &       & 0.630 & 0.666 & 0.638 & 0.553 \\
      & 1500  & 0.719 &       & 0.979 & 0.985 & 0.980 & 0.949 &       &       & 1500  & 0.513 &       & 0.910 & 0.906 & 0.885 & 0.816 \\
\cmidrule{1-8}\cmidrule{10-17}\end{tabular}}
%
\centering
  \noindent \caption{\small Empirical sizes and powers of $\tilde T_n^0$ and $\tilde T_n^\alpha$ in the case of IO contamination with $p=3\%$ and $\sigma_v^2=10$.} \label{GARCH:IO2}%
  \tabcolsep=2.3pt
 {\footnotesize
\begin{tabular}{lcccccccrlccccccc}
\cmidrule{1-8}\cmidrule{10-17}      &       &       & \multicolumn{5}{c}{$\tilde T^\A_n$}      &       &       &       &       & \multicolumn{5}{c}{$\tilde T^\A_n$} \\
\cmidrule{4-8}\cmidrule{13-17}      & $n$   & $\tilde T^0_n$ &       & 0.1   & 0.2   & 0.3   & 0.5   &       &       & $n$   & $\tilde T^0_n$ &       & 0.1   & 0.2   & 0.3   & 0.5 \\
\cmidrule{1-8}\cmidrule{10-17}Sizes & 500   & 0.053 &       & 0.028 & 0.030 & 0.037 & 0.038 &       & Sizes & 500   & 0.227 &       & 0.072 & 0.058 & 0.051 & 0.054 \\
$(w, \alpha_1, \beta_1)$ & 1000  & 0.044 &       & 0.040 & 0.046 & 0.044 & 0.050 &       & $(w, \alpha_1, \beta_1)$ & 1000  & 0.222 &       & 0.050 & 0.043 & 0.038 & 0.038 \\
=(0.5, 0.2, 0.4) & 1500  & 0.038 &       & 0.044 & 0.046 & 0.044 & 0.040 &       & =(0.5, 0.15, 0.8) & 1500  & 0.218 &       & 0.044 & 0.042 & 0.044 & 0.038 \\
\cmidrule{1-8}\cmidrule{10-17}$w:0.5\rightarrow0.8$ & 500   & 0.163 &       & 0.235 & 0.288 & 0.293 & 0.252 &       & $w:0.5\rightarrow0.2$ & 500   & 0.210 &       & 0.070 & 0.062 & 0.056 & 0.048 \\
      & 1000  & 0.227 &       & 0.549 & 0.646 & 0.654 & 0.583 &       &       & 1000  & 0.284 &       & 0.135 & 0.140 & 0.136 & 0.110 \\
      & 1500  & 0.324 &       & 0.786 & 0.856 & 0.846 & 0.802 &       &       & 2000  & 0.467 &       & 0.316 & 0.375 & 0.348 & 0.260 \\
\cmidrule{1-8}\cmidrule{10-17}$\alpha_1:0.2\rightarrow0.5$ & 500   & 0.144 &       & 0.322 & 0.454 & 0.498 & 0.492 &       & $\alpha_1:0.15\rightarrow0.05$ & 500   & 0.179 &       & 0.348 & 0.518 & 0.543 & 0.540 \\
      & 1000  & 0.258 &       & 0.820 & 0.912 & 0.914 & 0.890 &       &       & 1000  & 0.300 &       & 0.844 & 0.918 & 0.929 & 0.906 \\
      & 1500  & 0.409 &       & 0.968 & 0.991 & 0.992 & 0.986 &       &       & 1500  & 0.466 &       & 0.984 & 0.998 & 0.997 & 0.989 \\
\cmidrule{1-8}\cmidrule{10-17}$\beta_1:0.4\rightarrow0.6$ & 500   & 0.134 &       & 0.282 & 0.370 & 0.374 & 0.337 &       & $\beta_1:0.8\rightarrow0.7$ & 500   & 0.234 &       & 0.220 & 0.289 & 0.307 & 0.283 \\
      & 1000  & 0.242 &       & 0.784 & 0.875 & 0.874 & 0.815 &       &       & 1000  & 0.349 &       & 0.574 & 0.706 & 0.720 & 0.684 \\
      & 1500  & 0.432 &       & 0.960 & 0.984 & 0.984 & 0.962 &       &       & 1500  & 0.480 &       & 0.853 & 0.928 & 0.922 & 0.889 \\
\cmidrule{1-8}\cmidrule{10-17}\end{tabular}}
\end{table}%

\begin{table}[H]
\centering
  \caption{\small Empirical sizes and powers of $\tilde T_n^0$ and $\tilde T_n^\alpha$  in the case of AO contamination with $p=1\%$ and $\sigma_v^2=10$.}\label{GARCH:AO1}
  \tabcolsep=2.3pt
   {\footnotesize
\begin{tabular}{lcccccccrlccccccc}
\cmidrule{1-8}\cmidrule{10-17}      &       &       & \multicolumn{5}{c}{$\tilde T^\A_n$}      &       &       &       &       & \multicolumn{5}{c}{$\tilde T^\A_n$} \\
\cmidrule{4-8}\cmidrule{13-17}      & $n$   & $\tilde T^0_n$ &       & 0.1   & 0.2   & 0.3   & 0.5   &       &       & $n$   & $\tilde T^0_n$ &       & 0.1   & 0.2   & 0.3   & 0.5 \\
\cmidrule{1-8}\cmidrule{10-17}Sizes & 500   & 0.068 &       & 0.054 & 0.058 & 0.066 & 0.064 &       & Sizes & 500   & 0.058 &       & 0.057 & 0.060 & 0.058 & 0.058 \\
$(w, \alpha_1, \beta_1)$ & 1000  & 0.051 &       & 0.053 & 0.056 & 0.056 & 0.060 &       & $(w, \alpha_1, \beta_1)$ & 1000  & 0.046 &       & 0.046 & 0.049 & 0.048 & 0.047 \\
=(0.5, 0.2, 0.4) & 1500  & 0.052 &       & 0.055 & 0.066 & 0.066 & 0.068 &       & =(0.5, 0.15, 0.8) & 1500  & 0.052 &       & 0.054 & 0.052 & 0.054 & 0.052 \\
\cmidrule{1-8}\cmidrule{10-17}$w:0.5\rightarrow0.8$ & 500   & 0.232 &       & 0.289 & 0.315 & 0.307 & 0.252 &       & $w:0.5\rightarrow0.2$ & 500   & 0.159 &       & 0.140 & 0.122 & 0.112 & 0.098 \\
      & 1000  & 0.472 &       & 0.736 & 0.756 & 0.724 & 0.618 &       &       & 1000  & 0.320 &       & 0.310 & 0.271 & 0.239 & 0.198 \\
      & 1500  & 0.678 &       & 0.921 & 0.935 & 0.922 & 0.870 &       &       & 2000  & 0.780 &       & 0.840 & 0.786 & 0.701 & 0.562 \\
\cmidrule{1-8}\cmidrule{10-17}$\alpha_1:0.2\rightarrow0.5$ & 500   & 0.258 &       & 0.384 & 0.444 & 0.462 & 0.426 &       & $\alpha_1:0.15\rightarrow0.05$ & 500   & 0.257 &       & 0.299 & 0.316 & 0.311 & 0.277 \\
      & 1000  & 0.702 &       & 0.876 & 0.896 & 0.884 & 0.832 &       &       & 1000  & 0.691 &       & 0.766 & 0.760 & 0.735 & 0.662 \\
      & 1500  & 0.898 &       & 0.984 & 0.986 & 0.982 & 0.967 &       &       & 1500  & 0.934 &       & 0.960 & 0.954 & 0.936 & 0.891 \\
\cmidrule{1-8}\cmidrule{10-17}$\beta_1:0.4\rightarrow0.6$ & 500   & 0.246 &       & 0.348 & 0.377 & 0.370 & 0.312 &       & $\beta_1:0.8\rightarrow0.7$ & 500   & 0.202 &       & 0.205 & 0.206 & 0.196 & 0.180 \\
      & 1000  & 0.606 &       & 0.874 & 0.884 & 0.854 & 0.766 &       &       & 1000  & 0.547 &       & 0.620 & 0.606 & 0.580 & 0.497 \\
      & 1500  & 0.824 &       & 0.985 & 0.982 & 0.976 & 0.941 &       &       & 1500  & 0.851 &       & 0.908 & 0.892 & 0.862 & 0.768 \\
\cmidrule{1-8}\cmidrule{10-17}\end{tabular}}
\centering
  \caption{\small Empirical sizes and powers of $\tilde T_n^0$ and $\tilde T_n^\alpha$   in the case of AO contamination with $p=3\%$ and $\sigma_v^2=10$.}\label{GARCH:AO2}
  \tabcolsep=2.3pt
   {\footnotesize
   \begin{tabular}{lcccccccrlccccccc}
\cmidrule{1-8}\cmidrule{10-17}      &       &       & \multicolumn{5}{c}{$\tilde T^\A_n$}      &       &       &       &       & \multicolumn{5}{c}{$\tilde T^\A_n$} \\
\cmidrule{4-8}\cmidrule{13-17}      & $n$   & $\tilde T^0_n$ &       & 0.1   & 0.2   & 0.3   & 0.5   &       &       & $n$   & $\tilde T^0_n$ &       & 0.1   & 0.2   & 0.3   & 0.5 \\
\cmidrule{1-8}\cmidrule{10-17}Sizes & 500   & 0.157 &       & 0.096 & 0.092 & 0.094 & 0.092 &       & Sizes & 500   & 0.074 &       & 0.065 & 0.070 & 0.068 & 0.082 \\
$(w, \alpha_1, \beta_1)$ & 1000  & 0.139 &       & 0.088 & 0.085 & 0.086 & 0.092 &       & $(w, \alpha_1, \beta_1)$ & 1000  & 0.060 &       & 0.058 & 0.060 & 0.060 & 0.064 \\
=(0.5, 0.2, 0.4) & 1500  & 0.121 &       & 0.090 & 0.092 & 0.090 & 0.087 &       & =(0.5, 0.15, 0.8) & 1500  & 0.054 &       & 0.059 & 0.064 & 0.065 & 0.064 \\
\cmidrule{1-8}\cmidrule{10-17}$w:0.5\rightarrow0.8$ & 500   & 0.190 &       & 0.246 & 0.322 & 0.338 & 0.302 &       & $w:0.5\rightarrow0.2$ & 500   & 0.154 &       & 0.154 & 0.154 & 0.168 & 0.158 \\
      & 1000  & 0.289 &       & 0.647 & 0.754 & 0.748 & 0.685 &       &       & 1000  & 0.292 &       & 0.348 & 0.365 & 0.350 & 0.315 \\
      & 1500  & 0.411 &       & 0.874 & 0.923 & 0.924 & 0.887 &       &       & 2000  & 0.687 &       & 0.854 & 0.854 & 0.816 & 0.726 \\
\cmidrule{1-8}\cmidrule{10-17}$\alpha_1:0.2\rightarrow0.5$ & 500   & 0.193 &       & 0.358 & 0.482 & 0.523 & 0.496 &       & $\alpha_1:0.15\rightarrow0.05$ & 500   & 0.264 &       & 0.331 & 0.376 & 0.398 & 0.380 \\
      & 1000  & 0.504 &       & 0.852 & 0.912 & 0.907 & 0.876 &       &       & 1000  & 0.631 &       & 0.785 & 0.827 & 0.820 & 0.750 \\
      & 1500  & 0.782 &       & 0.982 & 0.992 & 0.992 & 0.984 &       &       & 1500  & 0.900 &       & 0.968 & 0.966 & 0.957 & 0.927 \\
\cmidrule{1-8}\cmidrule{10-17}$\beta_1:0.4\rightarrow0.6$ & 500   & 0.182 &       & 0.318 & 0.432 & 0.466 & 0.432 &       & $\beta_1:0.8\rightarrow0.7$ & 500   & 0.204 &       & 0.252 & 0.276 & 0.282 & 0.262 \\
      & 1000  & 0.368 &       & 0.807 & 0.876 & 0.874 & 0.818 &       &       & 1000  & 0.485 &       & 0.647 & 0.676 & 0.664 & 0.604 \\
      & 1500  & 0.566 &       & 0.968 & 0.982 & 0.984 & 0.968 &       &       & 1500  & 0.808 &       & 0.931 & 0.936 & 0.919 & 0.872 \\
\cmidrule{1-8}\cmidrule{10-17}\end{tabular}}
\end{table}%

Table \ref{tab:GARCH} reports the results for uncontaminated cases. It can be seen that each $\tilde T_n^\alpha$ achieves good sizes in all cases. One can also see that the empirical powers of the tests increase as the sample size increases and the tests yield good powers except for the case where $(w,\alpha_1,\beta_1)=(0.5,0.15,0.8)$ changes to $(0.2,0.15,0.8)$, say Case$^*$. In this case, the proposed test is observed to be less powerful, so $n=2000$ is considered only for the case. As in the i.i.d. cases above, $\tilde T_n^\alpha$ performs similarly to $\tilde T_n^0$ when $\alpha$ is close to 0 and also shows a decreasing trend in powers as $\alpha$ increases. As seen in the previous Case$^*$, $\tilde T^\A_n$ with $\A>0.3$ can show significant loss in power, thus we recommend not to use too large an $\A$.

Results for the IO contaminated cases are presented in Tables \ref{GARCH:IO1} and \ref{GARCH:IO2}. We first note that $\tilde T_n^0$ exhibits significant power losses whereas $\tilde T_n^\alpha$ with $\A>0$ maintains good powers also except for Case$^*$. In the case of $(w,\alpha_1,\beta_1)=(0.5,0.15,0.8)$, the sizes of $\tilde T_n^0$ is severely distorted when  $p=3\%$ and $\sigma_v^2=10$, but the proposed test shows no distortions. Although $\tilde T_n^0$ yields higher empirical sizes in this case, substantial power losses are observed in $\tilde T_n^0$.

Tables \ref{GARCH:AO1} and \ref{GARCH:AO2} summarize the results for the cases of the AO contamination.
We can also see the power losses of $\tilde T_n^0$, but not as large as the IO contaminated cases. This indicates that the score test is more affected by IO than by AO. Interestingly, $\tilde T_n^0$ is almost insensitive to AO in the case of $(w,\alpha_1,\beta_1)=(0.5,0.15,0.8)$. But, even in this cases, $\tilde T_n^\alpha$ outperforms $\tilde T_n^0$.

Overall, our simulation results strongly support the validity of the proposed test. In this section, we can see that our test is  sufficiently robust against outliers and the test with $\A$ close to 0 is as powerful as the score test when data is not contaminated. However, we also observe that $\tilde T^\A_n$ can suffer from power losses when a large $\A$ is employed. So, one should be careful not to use $\A$ that is too large. As mentioned in Remark \ref{ALPHA}, we recommend to use an $\A$ in [0.1,0.3] based on the simulation results.

\section{Real data analysis}
\begin{figure}[!t]
\includegraphics[height=0.45\textwidth,width=\textwidth]{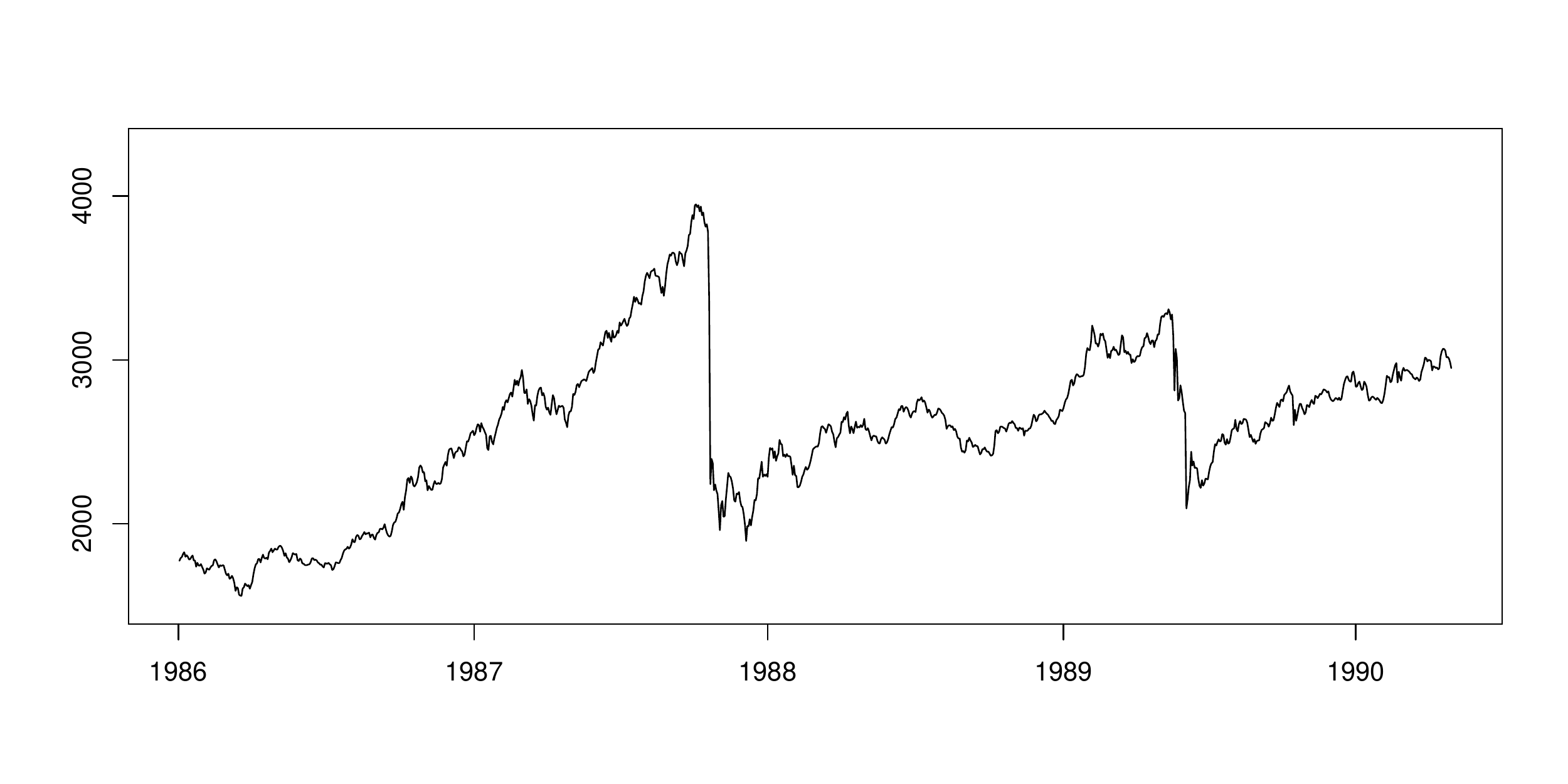}\vspace{-1cm}
\caption{\small Time series  plot of the Hang Seng index from Jan 2, 1986 to April 30, 1990}\label{HS}
\includegraphics[height=0.45\textwidth,width=1\textwidth]{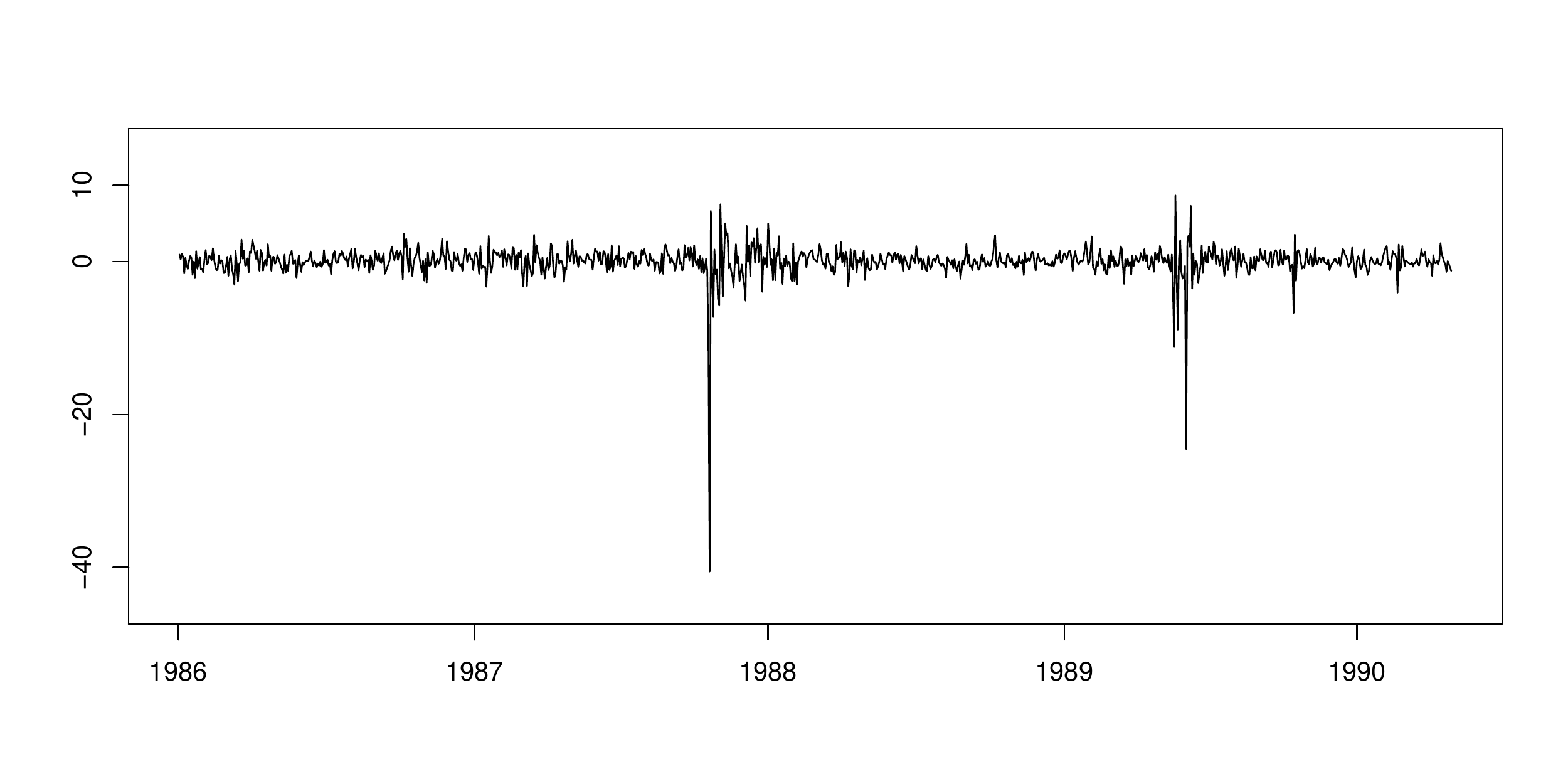}\vspace{-1cm}
\caption{\small Return series of the Hang Seng index from Jan 2, 1986 to April 30, 1990}\label{Return}
\end{figure}
In this section, we illustrate a real data application to the Hang Seng index in Hong Kong stock market. Our data consists of daily closing prices from Jan 2, 1986 to April 30, 1990. The index series $\{X_t\}$ and its return series $\{r_t\}$, where $X_t$ is the index value at time $t$ and $r_t=100\ log(X_t/X_{t-1})$, are displayed in Figures \ref{HS} and \ref{Return}, respectively. In Figure \ref{Return}, we can see some deviating observations which may interfere with correct statistical inferences. This data was previously analyzed by Lee and Song (2009), where they fitted GARCH(1,1) model to the data and estimated the model using MDPD estimator. In the present analysis, we also fit GARCH(1,1) model to the index data from 1986 to 1989 and examine whether there were parameter changes during the period. And then, based on each testing result, we calculate one-step-ahead out-of-sample forecasts for the conditional variance. A proper $\A$ for the data is chosen based on the forecasting results.

ML estimates of GARCH(1,1) model for the period from 1986 to 1989 are obtained as follows: $\hat w=0.112, \hat\A_1=0.282$, and $\hat\beta_1=0.740$.
Here, it should be noted that $\hat\A_1 +\hat\beta_1$ is greater than one, indicating that the parameters are estimated out of the stationary region of parameter space. One can guess that outlying observations unduly affected the ML estimation. As addressed in the studies such as Hillebrand (2005), parameter changes can also result in spuriously high estimates of $\A_1 +\beta_1$. Taking into account both possible cases, we perform the proposed test $\tilde T_n^\A$ for $\A \in \{0,0.1,0.2,\cdots, 1.0\}$,
 where $\tilde T_n^0$ is the score test. For comparison, we additionally conduct the following residual-based CUSUM test for parameter change:
 \[T_n^R:=\frac{1}{\sqrt{n}\hat{\tau}_n} \max_{1\leq k\leq n } \Big|\sum_{t=1}^k \hat{\ep}^2_{t} -\frac{k}{n}\sum_{t=1}^n\hat{\ep}_{t}^2\Big|,\]
where $\hat\ep_t$ denotes the residual in GARCH(1,1) model and $ \hat{\tau}_n^2=\frac{1}{n}\sum_{t=1}^n\hat{\ep}_{t}^4-\Big(\frac{1}{n}\sum_{t=1}^n\hat{\ep}_{t}^2\Big)^2$. Under $H_0$, $T_n^R$ converges in distribution to $\sup_{0\leq s \leq 1} |B^o(s)|$, where $\{B^o(s) | s\geq 0\}$ is the standard Brownian bridge (cf. Kulperger and Yu (2005) and Song and Kang (2018)). $T_n^R$ and its p-value are obtained to be 0.935 and 0.653, respectively. Hence, $T_n^R$ do not reject $H_0$.
\begin{table}[t]
  \centering
  {\small
  \tabcolsep=5pt
  \renewcommand{\arraystretch}{1.2}
  \caption{Test statistics [p-value] and estimated change point of the score test $\tilde T_n^0$ and $\tilde T_n^\A$}
    \begin{tabular}{cccccccccccc}
    \toprule
    $\A$       & 0     & 0.1        & 0.2        & 0.3        & 0.4        & 0.5        & 0.6        & 0.7        & 0.8        & 0.9        & 1.0 \\
    \midrule
    $\tilde T_n^\A$   & 0.681      & 3.069      & 3.755      & 4.051      & 4.256      & 4.369      & 4.288      & 4.102      & 3.887      & 3.679      & 3.491 \\
               & [0.964]    & [0.046]    & [0.015]    & [0.008]    & [0.005]    & [0.005]    & [0.005]    & [0.008]    & [0.011]    & [0.017]    & [0.023] \\
    chg.pt     & $\cdot$    & 575        & 575        & 575        & 568        & 568        & 568        & 568        & 568        & 568        & 568 \\
    \bottomrule
    \end{tabular}
  \label{tab:test}}
\end{table}%

Table \ref{tab:test} provides the test statistics, p-values, and estimated change-points. We first note that $\tilde T_n^\A$ with $\A>0$ produce p-values less than 0.05 whereas $\tilde T_n^0$ yields the p-value of 0.964. That is, for all $\A$ considered, the proposed tests reject $H_0$ but the score test do not reject. Recalling the simulation results that $\tilde T_n^0$ suffer from  power losses in the presence of outliers, we could presume that $\tilde T_n^0$ misses a significant parameter change. $\tilde T_n^\A$ with $\A\leq 0.3$ and $\A\geq 0.4$ locate 575 (May 4, 1988) and  568(April 25, 1988), respectively, as the change-points.
To detect further changes, we conduct $\tilde T_n^\A$ for each $\A>0$ using the binary segmentation method in Remark \ref{BS}, and one more change-point (Aug 17, 1987) is detected by $\tilde T_n^\A$ with $\A \leq 0.3$. The sub-period obtained by each $\tilde T_n^\A$ and estimation results are presented in Table \ref{tab:est}. From the table, one can see that all values of $\hat\A_1 +\hat\beta_1$ are less than one in all sub-periods. In particular, for every $\A$, $\A_1+\beta_1$ in the last sub-period is estimated to be smaller than in the previous sub-periods and $\hat w$ is obtained to be larger than before.

\begin{table}[h]
  \centering
{\small
\tabcolsep=7pt
  \caption{ML and MDPD estimates for each sub-period obtained by $\tilde T_n^\A$}
    \begin{tabular}{cccccccccccc}
    \toprule
    Period           & \multicolumn{3}{c}{1/2/86 - 12/29/89} &            &            &            &            &            &            &            &  \\
\cmidrule{2-4}               & $\hat w$   & $\hat\A_1$ & $\hat\beta_1$ &            &            &            &            &            &            &            &  \\
    MLE        & 0.112      & 0.282      & 0.740      &            &            &            &            &            &            &            &  \\
    \midrule
   Sub-Period            & \multicolumn{3}{c}{1/2/86 - 8/17/87} &            & \multicolumn{3}{c}{8/18/87 - 5/4/88} &            & \multicolumn{3}{c}{5/5/88 - 12/29/89} \\
\cmidrule{2-4}\cmidrule{6-8}\cmidrule{10-12}    MDPDE       & $\hat w$   & $\hat\A_1$ & $\hat\beta_1$ &            & $\hat w$   & $\hat\A_1$ & $\hat\beta_1$ &            & $\hat w$   & $\hat\A_1$ & $\hat\beta_1$ \\
    0.1        & 0.049      & 0.063      & 0.900      &            & 0.055      & 0.012      & 0.946      &            & 0.265      & 0.139      & 0.574 \\
    0.2        & 0.050      & 0.058      & 0.903      &            & 0.055      & 0.012      & 0.945      &            & 0.205      & 0.088      & 0.662 \\
    0.3        & 0.051      & 0.053      & 0.906      &            & 0.058      & 0.013      & 0.944      &            & 0.178      & 0.072      & 0.700 \\
    \midrule
     Sub-Period            & \multicolumn{3}{c}{1/2/86 - 4/25/88} &            & \multicolumn{3}{c}{4/26/88 - 12/29/89} &            &            &            &  \\
\cmidrule{2-4}\cmidrule{6-8}   MDPDE    & $\hat w$   & $\hat\A_1$ & $\hat\beta_1$ &            & $\hat w$   & $\hat\A_1$ & $\hat\beta_1$ &            &            &            &  \\
    0.4        & 0.062      & 0.020      & 0.931      &            & 0.136      & 0.051      & 0.764      &            &            &            &  \\
    0.5        & 0.063      & 0.021      & 0.929      &            & 0.118      & 0.042      & 0.793      &            &            &            &  \\
    0.6        & 0.065      & 0.022      & 0.927      &            & 0.110      & 0.040      & 0.802      &            &            &            &  \\
    0.7        & 0.067      & 0.023      & 0.925      &            & 0.108      & 0.040      & 0.802      &            &            &            &  \\
    0.8        & 0.069      & 0.025      & 0.922      &            & 0.108      & 0.042      & 0.799      &            &            &            &  \\
    0.9        & 0.071      & 0.026      & 0.919      &            & 0.108      & 0.043      & 0.794      &            &            &            &  \\
    1.0        & 0.072      & 0.028      & 0.916      &            & 0.109      & 0.045      & 0.789      &            &            &            &  \\
    \bottomrule
    \end{tabular}
  \label{tab:est}}
\end{table}%

Now, we calculate one-step-ahead out-of-sample forecasts of the conditional variance and compare forecasting performance of the models without and with parameter changes. Forecasting using the model with changes means that predicted values are obtained using the data after the last change-point. That is, letting $t_c$ be the last change-point, prediction of $\sigma^2_{t+1}$ is conducted using the observation $\{r_{t_c+1},\cdots,r_t\}$.  For $\A \leq 0.3$ and $\A\geq 0.4$, $t_c$ is 575 and 568, respectively. In the case of no change, $t_c$ is 0. For the purpose of comparison, we estimate the model without change using MLE. This situation describes that the data is analyzed without any robust methods. For the period from Jan 1990 to April 1990, total 80 observations, one-step-ahead predicted value, $\hat\sigma^2_{t,t+1}$, of the conditional variance is calculated as follows:
\[ {\hat\sigma}_{t,t+1}^2 = \hat w_t +\hat\A_t r_t^2 +\hat\beta_t \tilde{\sigma}_t^2(\hat w_t, \hat\A_t, \hat \beta_t),\]
where $\hat w_t, \hat\A_t$, and $\hat\beta_t$ are the estimates obtained using the data $\{r_{t_c+1},\cdots, r_t\}$ and $ \tilde{\sigma}_t^2(\hat w_t, \hat\A_t, \hat \beta_t)$ is the one recursively calculated as in (\ref{tilde1}). Since the true conditional variances  are unobservable, we use $r_{t+1}^2$ as a proxy of $\sigma^2_{t+1}$. The following root mean squared error (RMSE) is considered to evaluate forecasting performance:
\[ \sqrt{\frac{1}{80}\sum_{t=987}^{1066}  (r_{t+1}^2-  {\hat\sigma}_{t,t+1}^2)^2},\]
where $r_{988}$  and $r_{1067}$ are the return values at Jan 2, 1990 and April 30, 1990, respectively. Table \ref{tab:rmse} present the forecasting errors. One can see that the model with parameter change produces the smaller RMSE. In terms of forecasting performance, a proper $\A$ can be selected as 0.1, which produce smallest RMSE. Based on the estimation and the forecasting results, we can therefore conclude that the model with parameter change is better fitted to the data.

Our empirical findings support the usefulness of our proposed test. The proposed test can detect the parameter changes in the presence of deviating observations, whereas the score test and the residual-based CUSUM test miss the significant changes.  The parameters are estimated  comparatively differently in each sub-period divided by the proposed test, and
the models accommodating parameter change show better forecasting performances.  In such situation that seemingly outliers are included in a data set being suspected of having parameter changes, we expect that  our test can be a promising tool for detecting parameter change.

\begin{table}[t]
  \centering
   \tabcolsep=5pt
  \renewcommand{\arraystretch}{1.2}
  {\small
 \caption{One-step-ahead forecasting errors for the models without and with parameter changes}
 \begin{tabular}{ccccccccccccc}
\toprule
Estimator  & MLE        &            & \multicolumn{10}{c}{MDPDE with $\A$} \\
\cmidrule{4-13}$\A$       & 0          &            & 0.1        & 0.2        & 0.3        & 0.4        & 0.5        & 0.6        & 0.7        & 0.8        & 0.9        & 1.0 \\
\cmidrule{2-2}\cmidrule{4-13}Model      & \multicolumn{1}{l}{without change} &            & \multicolumn{10}{c}{with change} \\
\midrule
RMSE       & 2.336      &            & 2.125      & 2.127      & 2.131      & 2.138      & 2.143      & 2.145      & 2.147      & 2.147      & 2.147      & 2.148 \\
\bottomrule
\end{tabular}
  \label{tab:rmse}}
\end{table}%

\section{Concluding remark}
In this study, we proposed a robust test for parameter change using the DP divergence. Since the DP divergence includes KL divergence, our test can be viewed as a generalized version of the score test. Under regularity conditions, the limiting null distribution of the proposed test is established.
Our simulation results demonstrated  that the proposed test is robust to outliers whereas the score test is damaged by outliers.  In particular, like the estimators induced from DP divergence, our test with small $\alpha$ is also observed to maintain strong robustness with little loss in power relative to the score test. In the real data analysis, the usefulness of the proposed test is demonstrated by locating some change-points that are not detected by the score test and the residual-based CUSUM test.

The MDPD estimation procedure can be conveniently applied to various parametric models including time series models and multivariate models. As seen in Section \ref{Sec:3}, once such MDPDE is set up, our test procedure can be readily extended to these models. We leave these extensions as a possible topic of future study.

\section{Appendix}
In the present section, we provide the proofs of Theorems \ref{Thm1} and \ref{Thm3} for the case of $\A> 0$. The following lemma is very helpful in proving Theorem \ref{Thm1}.
\begin{lemma}\label{Lm1}Suppose that the assumptions {\bf A1}-{\bf A6} in Theorem \ref{Thm1} hold. Then, under $H_0$, we have
\begin{eqnarray*}\label{Lm1.1}
\max_{ 1\leq k\leq n} \frac{k}{n}\Big\|\paa H_{\A,k}(\bar\T_{\A,n,k})+J_\A\Big\|=o(1)\quad a.s.,
\end{eqnarray*}
where $\{\bar\T_{\A,n,k}\,|\, 1\leq k\leq n, n\geq1\}$ is any double array of $\Theta$-valued random vectors with $\|\bar\T_{\A,n,k}-\T_0\| \leq \|\hat\T_{\A,n}-\T_0\|$.
\end{lemma}
\begin{proof}
By {\bf A5}, we have
\begin{eqnarray}\label{mom.cond1}
\E \sup_{\T \in N(\T_0)} \big\|\paa l_\A(X;\T)-\paa l_\A(X;\T_0)\big\| <\infty.
\end{eqnarray}
Then, for any $\ep>0$, we can take a neighborhood $N_\ep(\T_0)$ such that
\begin{eqnarray}\label{mom.cond2}
\E \sup_{\T \in N_\ep(\T_0)} \big\|\paa l_\A(X;\T)-\paa l_\A(X;\T_0)\big\| <\ep
\end{eqnarray}
by decreasing the neighborhood in (\ref{mom.cond1}) to the singleton $\T_0$.
Since $\HT$ converges almost surely to $\theta_0$, we have that for sufficiently large $n$,
\begin{eqnarray*}
&&\max_{ 1\leq k\leq n} \frac{k}{n}\Big\|\paa H_{\A,k}(\bar\T_{\A,n,k})+J_\A\Big\|\\
&&\leq \max_{ 1\leq k\leq n} \frac{k}{n}\big\|\paa H_{\A,k}(\bar\T_{\A,n,k})-\paa H_{\A,k}(\T_0)\big\|
+ \max_{ 1\leq k\leq n} \frac{k}{n}\Big\|\paa H_{\A,k}(\T_0)+J_\A\Big\|\\
&&\leq\frac{1}{n} \sum_{i=1}^n \sup_{\theta\in N_{\ep}(\theta_0)} \| \paa\,l_\A(X_i;\T)-\paa\,l_\A(X_i;\T_0)\|
+ \max_{ 1\leq k\leq n} \frac{k}{n}\Big\|\paa H_{\A,k}(\T_0)+J_\A\Big\|\\
&&:=I_n +II_n\quad a.s.
\end{eqnarray*}
Due to (\ref{mom.cond2}), we can see that
\begin{eqnarray*}
\lim_{n\rightarrow\infty}I_n =
\E\sup_{\theta \in N_\ep(\theta_0)}\| \paa l_\A (X;\T)- \paa l_\A (X;\T_0)\| < \ep\quad a.s.
\end{eqnarray*}
Also, using the fact that $\|\paa H_{\A,n}(\T_0)+J_\A\|$ converges to zero almost surely, we have
\begin{eqnarray}\label{Lm1.2}
&&\max_{1\leq k \leq \sqrt{n}}  \frac{k}{n}\Big\|\paa H_{\A,k}(\T_0)+J_\A\Big\|
\leq \frac{1}{\sqrt{n}} \sup_n \Big\|\paa H_{\A,n}(\T_0)+J_\A\Big\|=o(1)\quad a.s.
\end{eqnarray}
and
\begin{eqnarray}\label{Lm1.3}
&&\max_{\sqrt{n} < k \leq n} \Big\|\paa H_{\A,k}(\T_0)+J_\A\Big\| =o(1)\quad a.s.,
\end{eqnarray}
which subsequently yield $II_n=o(1)$ a.s.  The lemma is therefore obtained.
\end{proof}

\noindent{\bf Proof of Theorem \ref{Thm1}}.\\
Recall in (\ref{paH_final}) that
\begin{eqnarray*}
\frac{[ns]}{\sqrt{n}}\pa_{\theta}{H}_{\A,[ns]}(\HT)
&=&\underbrace{\frac{[ns]}{\sqrt{n}}\pa_{\theta}{H}_{\A,[ns]}(\T_0)+\frac{[ns]}{n}\paa{H}_{\A,[ns]}(\T^*_{\A,n,s})
J_\A^{-1}\sqrt{n}\,\pa_{\T}{H}_{\A,n}(\T_0)}_{I_{\A,n}(s)}\\
&&+\underbrace{\frac{[ns]}{n}\paa {H}_{\A,[ns]}(\T^*_{\A,n,s})
J_\A^{-1}(B_{\A,n}+J_\A)\sqrt{n}(\HT-\T_0)}_{II_{\A,n}(s)}.
\end{eqnarray*}
Since $\E[\pa_\T l_\A(X;\T_0)]=0$ and $\{\pa_\T l_\A(X_i;\T_0)\}$ is a sequence of i.i.d. random vectors, it follows from the functional central limit theorem that
\begin{eqnarray*}
\frac{[ns]}{\sqrt{n}} \pa_\T H_{\A,[ns]}(\T_0)\stackrel {w}{\longrightarrow}K_{\alpha}^{1/2}B_d(s)~~in~~\mathbb {D}([0,1],\mathbb {R}^d ),
\end{eqnarray*}
where $\{B_d(s)|0\leq s\leq1\}$ is a $d$-dimensional standard Brownian motion. Thus, by the continuous mapping theorem, we have
\begin{eqnarray}\label{b.bridge}
 I^o_{\A,n}(s):=\frac{[ns]}{\sqrt{n}}\pa_{\T}H_{\A,[ns]}(\T_0)-\frac{[ns]}{n}\sqrt{n}\,\pa_{\theta}H_{\A,n}(\theta_0)
\stackrel{w}{\longrightarrow}
K_\A^{1/2}B^o_d(s)\quad \rm{in}\ \ \mathbb{D}\,\big( [0,1],\, \mathbb{R}^d\big).
\end{eqnarray}
Using Lemma \ref{Lm1} and the fact that  $\sqrt{n}\,\pa_{\theta}H_{\A,n}(\T_0) $ is $O_P(1)$, we can see that
\begin{eqnarray*}
&&\sup_{0\leq s \leq1}\big\|I_{\A,n}(s)-I_{\A,n}^o(s)\big\|\\
&&=\sup_{0\leq s \leq1}\Big\|\frac{[ns]}{n}\paa H_{\A,[ns]}(\T^*_{\A,n,s})J_\A^{-1}\,\sqrt{n}\,\pa_{\theta}H_{\A,n}(\T_0)
+\frac{[ns]}{n}\sqrt{n}\,\pa_{\theta}H_{\A,n}(\T_0)\Big\|\\
&&\leq \Big\|J_\A^{-1} \sqrt{n}\,\pa_{\theta}H_{\A,n}(\T_0) \Big\|\,\max_{ 1\leq k\leq n} \frac{k}{n}\Big\|\paa H_{\A,k}(\T^*_{\A,n,k})+J_\A\Big\|
=o_P(1),
\end{eqnarray*}
where $\T^*_{\A,n,k}$ denotes the one corresponding to $\T^*_{\A,n,s}$ when $[ns]=k$, which together with (\ref{b.bridge}) yields
\begin{eqnarray}\label{I1}
I_{\A,n}(s)
\stackrel{w}{\longrightarrow}
 K_\A^{1/2}B^o_d(s)\quad \rm{in}\ \ \mathbb{d}\,\mathbb{D}\big( [0,1],\, \mathbb{R}^d\big).
\end{eqnarray}
Next, note that by Lemma \ref{Lm1},
\begin{eqnarray}\label{II1}
\sup_{0\leq s\leq1} \frac{[ns]}{n} \|\paa H_{\A,[ns]}(\T^*_{\A,n,s}) \|
\leq\max_{ 1\leq k\leq n} \frac{k}{n} \|\paa {H}_{\A,k}(\T^*_{\A,n,k})+J_\A\|
+\|J_\A\|=O(1)\quad a.s.
\end{eqnarray}
and
\begin{eqnarray}\label{II2}
\|J_\A^{-1}(B_{\A,n} +J_\A)\| \leq \|J_\A^{-1}\|\max_{ 1\leq k\leq n} \frac{k}{n}\Big\| \paa H_{\A,k}(\T^*_{\A,n,k})+J_\A\Big\|=o(1)\quad a.s.
\end{eqnarray}
Further, using $\sqrt{n}\pa_\T H_{\A,n}(\T_0)=O_P(1)$ again and $B_{\A,n} +J_\A=o_P(1)$ obtained in (\ref{II2}), one can see from (\ref{hat.theta}) that
$\sqrt n(\hat\T_{\A,n}-\T_0)=O_P(1)$. Therefore, combing this, (\ref{II1}), and (\ref{II2}), we have that  $\sup_{0\leq s\leq1}\|II_{\A,n}(s)\|=o_P(1)$. This completes the proof.
\hfill{$\Box$}\vspace{0.2cm}\\



The following part is provided for Theorem \ref{Thm3} in Section \ref{Sec:3}. Without confusion, $H_{\A,k}(\T)$  is hereafter used to denote the one obtained by replacing $\tilde l_\A(X_t;\T)$ with $l_\A(X_t;\T)$ in (\ref{root}). $l_\A(X_t;\T)$  is defined in Remark \ref{Rm.l}.

Theorem \ref{Thm3} is shown by similar arguments to that of Theorem \ref{Thm1}. Fundamental lemmas for the proof are Lemmas \ref{Lm.G1} and \ref{Lm.G4}. The first lemma states that the functional central limit theorem also holds for $\{\pa_\T \tilde
l_\A(X_t;\T_0)\}$ and the latter one provides the result corresponding to Lemma \ref{Lm1}, which is also usefully used in proving Theorem \ref{Thm3}. Lemmas \ref{Lm.G2} and \ref{Lm.G3} are given to verify the aforementioned two lemmas, where we employ the technical results in  Francq and Zako\"{i}an (2004) and Lee and Song (2009) to prove the lemmas.


\begin{lemma}\label{Lm.G2} Suppose that the assumptions {\bf A1}-{\bf A4} in Theorem \ref{Thm3} hold. Then, under $H_0$, we have
\begin{eqnarray}\label{Lm.G2.0}
\frac{1}{\sqrt{n}}\sum_{t=1}^n\big\|\pa_{\T} l_\A(X_t;\T_0)-\pa_{\T} \tilde l_\A(X_t;\T_0)\big\|=o(1)\quad a.s.,
\end{eqnarray}
and, for some neighborhood $V_1(\T_0)$ of $\T_0$,
\begin{eqnarray}\label{Lm.G2.1}
\frac{1}{n}\sum_{t=1}^{n}\sup_{\theta \in V_1(\T_0)}\big\|\paa  l_\A(X_t;\theta)
-\paa  \tilde l_\A(X_t;\theta)\big\|&=& o(1)\quad a.s,
\end{eqnarray}
\begin{eqnarray}\label{Lm.G2.2}
\frac{1}{n}\sum_{t=1}^{n}\sup_{\theta \in V_1(\T_0)}\big\|\pa_\T  l_\A(X_t;\theta) \pa_{\T'}  l_\A(X_t;\theta)
-\pa_\T  \tilde l_\A(X_t;\theta) \pa_{\T'}  \tilde l_\A(X_t;\theta)\big\|&=& o(1)\quad a.s.
\end{eqnarray}
\end{lemma}
\begin{proof}
The proof of the part (iv) in Lee and Song (2009), page 337, implies the first two results in the lemma, and thus we omit the proofs of (\ref{Lm.G2.0}) and (\ref{Lm.G2.1}).
To verify (\ref{Lm.G2.2}), we introduce the following technical results in  Francq and Zako\"{i}an (2004):
\begin{eqnarray}
&&\sup_{\T\in\Theta}\big\{|\,\sigma_t^2-{\tilde{\sigma}}_t^2| \vee \|\pa_\T \tilde{\sigma}_t^2-\pa_\T\sigma_t^2\|\big\} \leq K\rho^t \ a.s.\quad \text{for all } t\geq1, \label{FZ.1}\\
&&\E \sup_{\T\in \Theta^*}\Big|\frac{1}{\sigma_t^2} \pa_{\T_i} \sigma_t^2\Big|^d < \infty
,\quad \E \sup_{\T \in \Theta^*}\Big|\frac{1}{\sigma_t^2}\pa_{\T_i \T_j} \sigma_t^2\Big|^d < \infty \quad \text{for all } d \in \mathbb{N}, \label{FZ.2}
\end{eqnarray}
where $K>0$ and $\rho \in (0,1)$ denote universal constants, which can take different values from line
to line, and  $\Theta^*$ is a compact set such that $\theta_0 \in  \Theta^* \subset \Theta^o$. The following moment result can be found in Lemma 1 in Lee and Song (2009): for all $d \in \mathbb{N}$,
\begin{eqnarray}\label{LS.1}
\E \sup_{\T\in V_1(\T_0) }\frac{X_t^{2d}}{\sigma_t^{2d}}  < \infty,
\end{eqnarray}
where  $V_1(\T_0)$ is a neighborhood of $\T_0$. We shall also use the results in page 340 in Lee and Song (2009). That is,
\begin{eqnarray}
&&|h_{\A}(\tilde{\sigma}_t^2)| \leq \Big(1+\frac{X_t^2}{\sigma_t^2}\Big),\quad \big|h_{\A}(\sigma_t^2)-h_{\A}({\tilde{\sigma}}_t^2)\big| \leq K \Big(1+\frac{X_t^2}{\sigma_t^2}+\frac{X_t^4}{ \sigma_t^4}\Big)\rho^t,\label{LS.2}\\
&&
\bigg|\left(\frac{1}{\sigma_t^2}\right)^{\frac{\A}{2}+1}-\left(\frac{1}{\tilde{\sigma}_t^2}
\right)^{\frac{\A}{2}+1}\bigg| \leq \frac{K\rho^t}{\sigma_t^2},\label{LS.3}
\end{eqnarray}
where \begin{eqnarray*}
h_{\A}(x)=-\frac{\A}{2\sqrt{1+\A}}+\frac{1+\A}{2}\left(1-\frac{X_t^2}{x}\right)\exp\left(-\frac{\A}{2}\frac{X_t^2}{x}\right).
\end{eqnarray*}
Here, we note that
\begin{eqnarray}\label{pa.l}
\pa_\T l_\A(X_t;\T)&=&h_{\A}(\sigma_t^2){\left(\frac{1}{\sigma_t^2}\right)}^{\frac{\A}{2}+1}\pa_\T\sigma_t^2.
\end{eqnarray}
Recall that $\Theta$ is also a compact subset in $(0,\infty)\times[0,\infty)^{p+q}$. Then, we have
\[\frac{1}{\sigma_t^2} \vee \frac{1}{\tilde \sigma_t^2}\leq \sup_{\T\in\Theta} \frac{1}{w} \leq K.\]
Using this, (\ref{FZ.1}), (\ref{LS.2}), and (\ref{LS.3}), we have
\begin{eqnarray*}
&&\big|\pa_{\T_i} l_\A(X_t;\T)-\pa_{\T_i} \tilde l_\A(X_t;\T)\big|\\
&&=\bigg|\big(h_{\A}(\sigma_t^2)-h_{\A}({\tilde{\sigma}}_t^2)\big){\left(\frac{1}{\sigma_t^2}
\right)}^{\frac{\A}{2}+1}\pa_{\T_i}\sigma_t^2
+h_{\A}({\tilde{\sigma}}_t^2)\bigg\{\left(\frac{1}{\sigma_t^2}\right)^{\frac{\A}{2}+1}-\left(\frac{1}{\tilde{\sigma}_t^2}
\right)^{\frac{\A}{2}+1}\bigg\} \pa_{\T_i}\sigma_t^2\\
&&
\quad+h_{\A}(\tilde{\sigma}_t^2)\left(\frac{1}{\tilde{\sigma}_t^2}
\right)^{\frac{\A}{2}+1}\big(\pa_{\T_i}\sigma_t^2-\pa_{\T_i}\tilde{\sigma}_t^2\big)\bigg|\\
&&\leq K\Big( 1+\frac{X_t^2}{\sigma_t^2}+\frac{X_t^4}{\sigma_t^4}\Big)\Big|\frac{1}{\sigma_t^2}\pa_{\T_i} \sigma_t^2\Big|\rho^t
+ K\Big( 1+\frac{X_t^2}{\sigma_t^2}\Big)\Big|\frac{1}{\sigma_t^2}\pa_{\T_i} \sigma_t^2\Big|\rho^t
+ K\Big( 1+\frac{X_t^2}{\sigma_t^2}\Big)\rho^t   \\
&&\leq K\Big( 1+\frac{X_t^2}{\sigma_t^2}+\frac{X_t^4}{\sigma_t^4}\Big)\Big(1+\Big|\frac{1}{\sigma_t^2}\pa_{\T_i} \sigma_t^2\Big|\Big)\rho^t :=K P_{t,i}(\T)\rho^t.
\end{eqnarray*}
Since $\big|\pa_{\T_i} l_\A(X_t;\T)\big| \leq K P_{t,i}(\T)$, we can also see from the above that
\begin{eqnarray*}
\big|\pa_{\T_i} \tilde l_\A(X_t;\T)\big| \leq \big|\pa_{\T_i} l_\A(X_t;\T)\big| + K P_{t,i}(\T)\leq   K P_{t,i}(\T),
\end{eqnarray*}
and thus, we have
\begin{eqnarray*}
&&\big|\pa_{\T_i}  l_\A(X_t;\theta) \pa_{\T_j}  l_\A(X_t;\theta)
-\pa_{\T_i}  \tilde l_\A(X_t;\theta) \pa_{\T_j}  \tilde l_\A(X_t;\theta)\big|\\
&&\leq
\big|\pa_{\T_i}  l_\A(X_t;\theta)\big| \big|\pa_{\T_j}  l_\A(X_t;\theta)- \pa_{\T_j} \tilde l_\A(X_t;\theta)\big|
+
\big|\pa_{\T_j}  \tilde l_\A(X_t;\theta)\big| \big|\pa_{\T_i}  l_\A(X_t;\theta)- \pa_{\T_i} \tilde l_\A(X_t;\theta)\big| \\
&& \leq K P_{t,i}(\T) P_{t,j}(\T)\rho^t.
\end{eqnarray*}
Using the Cauchy-Schwarz inequality  with the first result in (\ref{FZ.2}) and (\ref{LS.1}), we can see that for all $d \in \mathbb{N}$,
\begin{eqnarray}\label{m.P}
\E \sup_{\T \in V_1(\T_0)} P_{t,i}^d(\T) <\infty.
\end{eqnarray} Therefore, since
\begin{eqnarray*}
\sum_{t=1}^\infty P\Big( \rho^t \sup_{\T \in V_1(\T_0)} P_{t,i}(\T) P_{t,j}(\T) >\epsilon \Big)
&\leq &\frac{1}{\epsilon} \sum_{t=1}^\infty  \rho^t \sqrt{\E \sup_{\T \in V_1(\T_0)} P^2_{t,i}(\T)}\sqrt{\E \sup_{\T \in V_1(\T_0)} P^2_{t,j}(\T)}\\
&<&\infty,
\end{eqnarray*}
$\rho^t\sup_{\T \in V_1(\T_0)} P_{t,i}(\T) P_{t,j}(\T)$ converges to zero with probability one. Hence, (\ref{Lm.G2.2}) is asserted from the Ces\`{a}ro lemma.
\end{proof}


\begin{lemma}\label{Lm.G3} Suppose that the assumptions {\bf A1}-{\bf A4} in Theorem \ref{Thm3} hold. Then, under $H_0$, we have that for some neighborhood $V_2(\T_0)$ of $\T_0$,
\begin{eqnarray*}\label{K}
\E\sup_{\theta \in V_2(\T_0)}\big\|\pa_{\T'}  l_\A(X_t;\T)\pa_{\T}  l_\A(X_t;\T)\big\|  <\infty\quad\mbox{and}\quad \E\sup_{\theta \in V_2(\T_0)}\big\|\paa  l_\A(X_t;\T)\big\|  <\infty.
\end{eqnarray*}
\end{lemma}
\begin{proof}
Since $\big|\pa_{\T_i} l_\A(X_t;\T)\, \pa_{\T_j} l_\A(X_t;\T)\big|\leq K P_{t,i}(\T) P_{t,j}(\T)$, the first result is obtained by (\ref{m.P}) and   the Cauchy-Schwarz inequality.
A straightforward calculation shows that
\begin{eqnarray*}
\paa l_\A(X_t;\T)=
h_{\A}(\sigma_t^2){\left(\frac{1}{\sigma_t^2}\right)}^{\frac{\A}{2}+1}\paa \sigma_t^2 +m_{\A}(\sigma_t^2){\left(\frac{1}{\sigma_t^2}\right)}^{\frac{\A}{2}+2}
\pa_\T \sigma_t^2\,\pa_{\T'}\sigma_t^2,
\end{eqnarray*}
where
\begin{eqnarray*}
m_{\A}(x)=\frac{\A(2+\A)}{4\sqrt{1+\A}}-\frac{1+\A}{2}\left\{1+\frac{\A}{2}-(2+\A)\frac{X_t^2}{x}+\frac{\A}{2}\frac{X_t^4}{x^2}\right\}\exp\left(-\frac{\A}{2}\frac{X_t^2}{x}\right).
\end{eqnarray*}
The second one can also be readily shown by using (\ref{FZ.2}) and (\ref{LS.1}), so we omit the proof for brevity.
\end{proof}

\begin{lemma}\label{Lm.G1} Suppose that the assumptions {\bf A1}-{\bf A4} in Theorem \ref{Thm3} hold. Then, under $H_0$, we have
\begin{eqnarray*}
\frac{1}{\sqrt{k(\A)}} J_{1,\A}^{-1/2}\frac{[ns]}{\sqrt{n}} \pa_\T \tilde
H_{\alpha,[ns]}(\theta_0)\stackrel {w}{\longrightarrow}B_D(s)~~in~~\mathbb {D}([0,1],\mathbb {R}^D ),
\end{eqnarray*}
where $B_D$ is a standard $D$-dimensional Brownian motion and $D=p+q+1$.
\end{lemma}
\begin{proof}
By simple algebra, we have $\E [h_\A(\sigma_t^2(\T_0)]=0$, and hence $\{\pa_\T l_\A(X_t;\T_0) \}$ becomes a
martingale difference process. Since $\{\pa_\T l_\A(X_t;\T_0) \}$ is also strictly stationary and ergodic, we can obtain that by the functional central limit theorem for martingale difference,
\begin{eqnarray}\label{FCLT}
\frac{1}{\sqrt{k(\A)}} J_{1,\A}^{-1/2}\frac{[ns]}{\sqrt{n}} \pa_\T
H_{\alpha,[ns]}(\theta_0)\stackrel {w}{\longrightarrow} B_D(s)~~in~~\mathbb {D}([0,1],\mathbb {R}^D ).
\end{eqnarray}
Due to (\ref{Lm.G2.0}), we also have
\begin{eqnarray*}
\sup_{0\leq s\leq 1} \frac{[ns]}{\sqrt{n}}\big\|\pa_\T H_{\A,[ns]}(\T_0)-\pa_\T \tilde H_{\A,[ns]}(\T_0)\big\| \leq
\frac{1}{\sqrt{n}}\sum_{t=1}^n\big\|\pa_{\T} l_\A(X_t;\T_0)-\pa_{\T} \tilde l_\A(X_t;\T_0)\big\|=o(1)\quad a.s.,
\end{eqnarray*}
which together with (\ref{FCLT}) ensures the lemma.
\end{proof}

\begin{lemma}\label{Lm.G4}Suppose that the assumptions {\bf A1}-{\bf A4} in Theorem \ref{Thm3} hold. Then, under $H_0$, we have
\begin{eqnarray*}
\max_{ 1\leq k\leq n} \frac{k}{n}\big\|\paa \tilde{H}_{\A,k}(\T^*_{\A,n,k})+J_\A\big\|=o_P(1).
\end{eqnarray*}
where $\{\T^*_{\A,n,k}\,|\, 1\leq k\leq n, n\geq1\}$ is any double array of $\Theta$-valued random vectors with $\|\T^*_{\A,n,k}-\T_0\| \leq \|\hat\T_{\A,n}-\T_0\|$.
\end{lemma}
\begin{proof}
For any $\ep>0$, using a similar argument in (\ref{mom.cond2}) together with Lemma \ref{Lm.G3}, one can take a neighborhood $N_\ep(\T_0)$ of $\T_0$ such that
\begin{eqnarray}\label{Lm.G4.2}
\E\sup_{\theta \in N_\ep(\T_0)}\big\| \paa l_{\A} (X_t,\theta)- \paa l_{\A} (X_t,\theta_0)\big\| < \ep.
\end{eqnarray}
Let $V_\ep(\T_0)=V_1(\T_0)\cap N_\ep(\T_0)$, where $V_1(\T_0)$ is the one given in Lemma \ref{Lm.G2}.
 Since $\HT$ converges almost surely to $\theta_0$, we have that for sufficiently large $n$,
\begin{eqnarray*}
&&\max_{ 1\leq k\leq n} \frac{k}{n}\big\|\paa \tilde{H}_{\A,k}(\T^*_{\A,n,k})+J_\A\big\|\\
&&\leq \max_{ 1\leq k\leq n} \frac{k}{n}\big\|\paa \tilde{H}_{\A,k}(\T^*_{\A,n,k})-\paa H_{\A,k}(\T^*_{\A,n,k})\big\|
+ \max_{ 1\leq k\leq n} \frac{k}{n}\big\|\paa H_{\A,k}(\T^*_{\A,n,k})-\paa H_{\A,k}(\T_0)\big\|\\
&&\hspace{0.5cm}
+ \max_{ 1\leq k\leq n} \frac{k}{n}\big\|\paa H_{\A,k}(\T_0)+J_\A\big\|\\
&&\leq \frac{1}{n} \sum_{t=1}^n \sup_{\T\in V_\ep(\T_0)} \big\| \paa\,\tilde{l}_\A(X_t;\T)-\paa\,l_\A(X_t;\T)\big\|
+\frac{1}{n} \sum_{t=1}^n \sup_{\theta\in V_\ep(\T_0)} \big\| \paa\,l_\A(X_t;\T)-\paa\,l_\A(X_t;\T_0)\big\|\\
&&\hspace{0.5cm}
+ \max_{ 1\leq k\leq n} \frac{k}{n}\big\|\paa H_{\A,k}(\T_0)+J_\A\big\|\\
&&:=I_n +II_n+ III_n\quad a.s.
\end{eqnarray*}
First, one can see that  $I_n=o(1)$ a.s. by (\ref{Lm.G2.1}).  Using (\ref{Lm.G4.2}) and the ergodicity of $\{l_\A (X_t;\T)\}$, we also have
\begin{eqnarray*}
\lim_{n\rightarrow\infty}II_n =
\E\sup_{\theta \in V_\ep(\T_0)}\big\| \paa l_\A(X_t;\T)- \paa l_\A(X_t;\T_0)\big\| < \ep\quad a.s.
\end{eqnarray*}
Finally, observe that $\|\paa H_{\A,n}(\T_0)+J_\A\|$ converges to zero almost surely by the ergodic theorem.  In the same way as in (\ref{Lm1.2}) and (\ref{Lm1.3}), it can be shown that $III_n=o(1)$ a.s., which completes the proof.
\end{proof}
\newpage
\noindent{\bf Proof of Theorem \ref{Thm3}}.\\
We first show in (\ref{paH_final2}) that
\begin{align}\label{til.I1}
 \tilde I_{\A,n}(s)&:=\frac{[ns]}{\sqrt{n}}\pa_{\theta}\tilde{H}_{\A,[ns]}(\T_0)+\frac{[ns]}{n}\paa \tilde{H}_{\A,[ns]}(\T^*_{\A,n,s})J_\A^{-1}\,\sqrt{n}\,\pa_{\theta}\tilde{H}_{\A,n}(\T_0)\nonumber\\
&\stackrel{w}{\longrightarrow}
 \sqrt{k(\A)} J_{1,\A}^{1/2}B^o_D(s)\quad \rm{in}\ \ \mathbb{D}\,\big( [0,1],\, \mathbb{R}^D\big).
\end{align}
Due to Lemma \ref{Lm.G1}, we have
\[\tilde I^o_{\A,n}(s):=\frac{[ns]}{\sqrt{n}}\pa_{\theta}\tilde{H}_{\A,[ns]}(\T_0)-\frac{[ns]}{n} \sqrt{n}\,\pa_{\theta}\tilde{H}_{\A,n}(\theta_0)
\stackrel{w}{\longrightarrow}
 \sqrt{k(\A)} J_{1,\A}^{1/2}B^o_D(s)\quad \rm{in}\ \ \mathbb{D}\,\big( [0,1],\, \mathbb{R}^D\big).\]
Observe that $\sqrt{n}\,\pa_{\theta}\tilde{H}_{\A,n}(\T_0)=O_P(1)$ by Lemma \ref{Lm.G1} with $s=1$. Then, it follows from Lemma \ref{Lm.G4} that
\begin{align*}
\sup_{0\leq s \leq1}\big\| \tilde I_{\A,n}(s)-\tilde I^o_{\A,n}(s)\big\|
&=\sup_{0\leq s \leq1}\Big\|\frac{[ns]}{n}\paa \tilde{H}_{\A,[ns]}(\T^*_{\A,n,s})J_\A^{-1}\,\sqrt{n}\,\pa_{\theta}\tilde{H}_{\A,n}(\T_0)
+\frac{[ns]}{n}\sqrt{n}\,\pa_{\theta}\tilde{H}_{\A,n}(\T_0)\Big\|\\
&\leq \big\|J_\A^{-1} \sqrt{n}\,\pa_{\theta}\tilde{H}_{\A,n}(\T_0) \big\|\,\max_{ 1\leq k\leq n} \frac{k}{n}\big\|\paa \tilde{H}_{\A,k}(\T^*_{\A,n,k})+J_\A\big\|=o(1)\quad a.s.,
\end{align*}
where $\theta^*_{\A,n,k}$ denotes $\theta^*_{\A,n,k/n}$, and thus (\ref{til.I1}) is obtained.

Next, note from Proposition \ref{Prop1} that $\sqrt{n}(\hat\T_{\A,n}-\T_0)=O_p(1)$. By Lemma \ref{Lm.G4}, we have
\[\sup_{0\leq s\leq1} \frac{[ns]}{n} \big\|\paa \tilde{H}_{\A,[ns]}(\T^*_{\A,n,s}) \big\|
\leq\max_{ 1\leq k\leq n} \frac{k}{n} \big\|\paa \tilde{H}_{\A,k}(\T^*_{\A,n,k})+J_\A\big\|
+\big\|J_\A\big\|=O_P(1),\]
and
\begin{align*}
\big\|J_\A^{-1}(B_{\A,n} +J_\A)\sqrt n(\hat\T_{\A,n}-\T_0)\big\| &\leq \big\|J_\A^{-1}\big\|\max_{ 1\leq k\leq n} \frac{k}{n}\big\| \paa \tilde{H}_{\A,k}(\T^*_{\A,n,k})+J_\A\big\|\big\|\sqrt n(\hat\T_{\A,n}-\T_0)\big\|\\
& =o_P(1),
\end{align*}
and consequently,
\begin{eqnarray*}
\sup_{0\leq s\leq1} \bigg\|\frac{[ns]}{\sqrt{n}}\pa_{\theta}\tilde{H}_{\A,[ns]}(\hat\T_{\A,n})-\tilde I_{\A,n}(s)\bigg\|=o_P(1),
\end{eqnarray*}
which together (\ref{til.I1}) establishes the theorem. \hfill{$\Box$}\vspace{0.2cm}\\

\begin{lemma}\label{Lm.G6}Suppose that the assumptions {\bf A1}-{\bf A4} in Theorem \ref{Thm3} hold. Then, under $H_0$, we have
\begin{eqnarray*}
\frac{1}{n}\sum_{t=1}^n \pa_\T \tilde l_\A(X_t;\hat \T_\A) \pa_{\T'} \tilde l_\A(X_t;\hat \T_\A) \stackrel{P}{\longrightarrow} \E\big[\pa_\T l_\A(X;\T_0)\pa_{\T'} l_\A(X;\T_0)\big].
\end{eqnarray*}
\end{lemma}
\begin{proof}
Using the first result in Lemma \ref{Lm.G3}, we can also take a neighborhood $N_\ep(\theta_0)$ such that
\begin{eqnarray}\label{Lm.G6.1}
&&\lim_{n\rightarrow\infty}\frac{1}{n}\sum_{t=1}^n\sup_{\theta\in
N_\ep(\theta_0)}\big\| \pa_{\theta}
l_\A(X_t;\theta)\,\pa_{\theta'} l_\A(X_t;\theta)- \pa_{\theta}
l_\A(X_t;\theta_0)\,\pa_{\theta'} l_\A(X_t;\theta_0)\big\| \nonumber\\
&&=
\E\sup_{\theta \in N_\ep(\theta_0)}\big\| \pa_{\theta}
l_\A(X_t;\theta)\,\pa_{\theta'} l_\A(X_t;\theta)- \pa_{\theta}
l_\A(X_t;\theta_0)\,\pa_{\theta'} l_\A(X_t;\theta_0)\big\| < \ep\quad a.s.
\end{eqnarray}
Since $\hat\theta_n$ converges to $\theta_0$ almost surely, the lemma follows from (\ref{Lm.G2.2}), (\ref{Lm.G6.1}) and the ergodic theorem.
\end{proof}
\noindent{\bf Acknowledgments}\\
This research was supported by Basic Science Research Program through the National Research Foundation of Korea (NRF) funded by the
Korean government (MSIT)(NRF-2016R1C1B1015963), (ME)(NRF-2019R1I1A3A01056924)(J.Song) and (MSIT)(NRF-2019R1G1A1099809)(J.Kang).

\vspace{1.cm}
\noindent{\bf References}
 \begin{description}
\item Aue, A. and Horváth, L. (2013). Structural breaks in time series. {\it Journal of Time Series Analysis} {\bf 34}, 1–16.
\item Basu, A., Harris, I. R., Hjort, N. L. and Jones, M. C. (1998). Robust and efficient estimation by minimizing a density power divergence. {\it Biometrika} {\bf 85}, 549-559.
\item Basu, A., Mandal, A., Martin, N. and Pardo, L. (2013). Testing statistical hypotheses based on the density power divergence. {\it Annals of the Institute of Statistical Mathematics} {\bf 65}, 319-348.
\item Basu, A., Mandal, A., Martin, N. and Pardo, L. (2016).Generalized Wald-type tests based on minimum density power divergence estimators. {\it Statistics} {\bf 50}, 1-26.
\item Batsidis, A., Horváth, L., Martín, N., Pardo, L. and Zografos, K. (2013). Change-point detection in multinomial data using phi-divergence test statistics. {\it Journal of Multivariate Analysis} {\bf 118}, 53-66.
\item Berkes, I. and Horv\'{a}th, L. (2004). The efficiency of the estimators of the parameters in GARCH processes.  {\it The Annals of Statistics} {\bf 32}, 633-665.
\item Berkes, I., Horv\'{a}th, L. and Kokoszka, P. (2004). Testing for parameter constancy in GARCH(p,q) models. {\it Statistics \& Probability Letters} {\bf 4}, 263-273.
\item Bougerol, P. and Picard, N. (1992). Stationarity of GARCH processes and of some nonnegative time series. {\it Journal of Econometrics} {\bf 52}, 115–127.
 \item  Durio, A. and Isaia, E.D. (2011). The minimum density power divergence approach in building robust regression models. {\it Informatica} {\bf 22}, 43--56.
\item Fearnhead, P. and Rigaill, G. (2018). Changepoint detection in the presence of outliers. {\it Journal of the American Statistical Association}, 1-15.
\item Ferguson, T.S. (1996). {\it A Course in Large Sample Theory}. Chapman \& Hall/CRC, New York.
 \item Francq, C. and Zako\"{\i}an, J.-M. (2004). Maximum likelihood estimation of pure GARCH and ARMA-GARCH processes. {\it Bernoulli} {\bf 10}, 605-637.
 \item Fujisawa, H. and Eguchi, S. (2006). Robust estimation in the normal mixture model. {\it  Journal of Statistical Planning and Inference} {\bf 136}, 3989–4011.
\item Fujisawa, H., and Eguchi, S. (2008). Robust parameter estimation with a small bias against heavy contamination. {\it Journal of Multivariate Analysis} {\bf 99}, 2053-2081.
\item Ghosh, A. and Basu, A. (2017). The minimum S-divergence estimator under continuous models: the Basu-Lindsay approach. {\it Statistical Papers} {\bf 58}, 341-372.
\item Ghosh, A., Basu, A. and  Pardo, L. (2015). On the robustness of a divergence based test of simple statistical hypotheses. {\it  Journal of Statistical Planning and Inference} {\bf 161}, 91-108.
\item Ghosh, A., Mandal, A., Martín, N. and Pardo, L. (2016). Influence analysis of robust Wald-type tests. {\it Journal of Multivariate Analysis} {\bf 147}, 102-126.
\item Hillebrand, E. (2005). Neglecting parameter changes in GARCH models. {\it Journal of Econometrics} {\bf 129}, 121-138.
 \item Horv\'{a}th, L. and Parzen, E. (1994). Limit theorems for Fisher-score change processes. {\it Change-point problems, IMS Lecture Notes.Monograph Series} {\bf 23}, 157-169.
\item Horváth, L. and Rice, G. (2014). Extensions of some classical methods in change point analysis. {\it TEST} {\bf 23}, 219-255.
\item Kang, J. and Lee, S. (2014). Minimum density power divergence estimator for Poisson AR models. {\it  Computational Statistics \& Data Analysis} {\bf 80}. 44-56.
\item Kang, J. and Song, J. (2015). Robust parameter change test for Poisson autoregressive models. {\it Statistics \& Probability Letters} {\bf 104}, 14-21.
\item Kulperger, R. and Yu, H. (2005). High moment partial sum processes of residuals in GARCH models and their applications. {\it The Annals of Statistics} {\bf 33}, 2395–2422.
\item Lee, S. and Na, O. (2005). Test for parameter change based on the estimator minimizing density-power divergence measures. {\it Annals of the Institute of Statistical Mathematics} {\bf 57}, 553-573.
\item Lee, S. and Song, J. (2009). Minimum density power divergence estimator for GARCH models. {\it TEST} {\bf 18}, 316-341.
\item Maronna, R. A., Martin, R. D. and Yohai, V. J.(2006). {\it Robust Statistics: Theory and Methods}. Wiley.
\item Pardo, L. (2006). {\it Statistical Inference Based on Divergence Measures}, Chapman and Hall/CRC.
\item Robbins, M., Gallagher, C., Lund, R. and Aue, A. (2011). Mean shift testing in correlated data. {\it Journal of Time Series Analysis} {\bf32}, 498-511.
\item Song, J. (2017). Robust estimation of dispersion parameter in discretely observed diffusion processes. {\it Statistica Sinica} {\bf 27}, 373-388.
\item Song, J. (2020). Robust test for dispersion parameter change in discretely observed diffusion processes.{\it  Computational Statistics \& Data Analysis} {\bf 142}, 106832
\item Song, J. and Kang, J. (2018). Parameter change tests for ARMA–GARCH models. {\it Computational Statistics \& Data Analysis} {\bf 121}, 41-56.
\item Tsay, R. S. (1988). Outliers, level shifts, and variance changes in time series. {\it Journal of Forecasting} {\bf 7}, 1-20.
\item Warwick, J., (2005). A data-based method for selecting tuning parameters in minimum distance estimators. {\it  Computational Statistics \& Data Analysis} {\bf 48}, 571--585.

\end{description}

\end{document}